\documentclass[12pt]{article}
\usepackage{enumerate}

\usepackage{amsthm}
\usepackage{bbm}
\usepackage{amsfonts}
\usepackage{mathrsfs}
\usepackage{amsmath}
\usepackage{fancyhdr}
\usepackage{hyperref}
\usepackage{color} 
\usepackage{mathtools}
\usepackage{stackrel,amssymb}
\usepackage{cancel}

\usepackage{float}

\usepackage{enumitem}
\usepackage{graphicx}
\usepackage{titlesec}
\usepackage{bigints}
\usepackage{relsize,exscale}
\usepackage{romannum}
\usepackage{multirow}
\usepackage{subcaption}
\usepackage[nottoc,numbib]{tocbibind}
\usepackage{tikz}
\usepackage{subcaption}
\usetikzlibrary{decorations.markings}
\usetikzlibrary{arrows}

\usetikzlibrary{trees}

\usepackage{arxiv}
\usepackage{hyperref} 
\usepackage{bm}
\usepackage{xcolor}
\usepackage{framed}
\usepackage{cleveref}

\theoremstyle{plain}
\newtheorem{theorem}{Theorem}
\newtheorem{Lemma}{Lemma}   
\newtheorem{prop}{Proposition}   

\theoremstyle{definition}

\theoremstyle{newremark}
\newtheorem{remark}{Remark}

\newtheorem{claim}{Claim}

\newcommand{\s}{\sigma}

\newcommand{\x}{\mathbf{x}}

\renewcommand{\a}{\alpha}
\newcommand{\e}{\epsilon}

\fancyheadoffset{0in}

\title{Linear-Quadratic Stochastic Differential Games on  Random Directed Networks}
\fancypagestyle{plain}{%
  \fancyhead{} 
}

\author{
\and  Yichen Feng\thanks{Department of Statistics and Applied Probability, South Hall, University of California, Santa Barbara, CA 93106, USA (E-mail: \href{mailto:feng@pstat.ucsb.edu}{feng@pstat.ucsb.edu}).} 
  \and Jean-Pierre Fouque\thanks{Department of Statistics and Applied Probability, South Hall, University of California, Santa Barbara, CA 93106, USA (E-mail: \href{mailto:fouque@pstat.ucsb.edu}{fouque@pstat.ucsb.edu}). Work supported by NSF grant DMS-1814091.} 
 \and Tomoyuki Ichiba\thanks{Department of Statistics and Applied Probability, South Hall, University of California, Santa Barbara, CA 93106, USA (E-mail: \href{mailto:ichiba@pstat.ucsb.edu}{ichiba@pstat.ucsb.edu}). Work supported by NSF grant DMS-1615229.} 
  }
\date{\vspace{-5ex}}

\begin{document}

\maketitle

\pagenumbering{arabic}

\begin{abstract}

The study of linear-quadratic stochastic differential games on directed networks was initiated in Feng, Fouque \& Ichiba \cite{fengFouqueIchiba2020linearquadratic}. In that work, the game on a directed chain with finite or infinite players was defined as well as the game on a deterministic directed tree, and their Nash equilibria were computed. The current work continues the analysis by first developing a random directed chain structure by assuming the interaction between every two neighbors is random. We solve explicitly for an open-loop Nash equilibrium for the system and we find that the dynamics under equilibrium is an infinite-dimensional Gaussian process described by a Catalan Markov chain introduced in \cite{fengFouqueIchiba2020linearquadratic}. The discussion about stochastic differential games is extended to a random two-sided directed chain and a random directed tree structure. 
\end{abstract}




\section{Introduction}
Stochastic differential games on networks have been studied widely with great interest in recent years. The present paper about stochastic differential games on random directed networks is a continuation of our work in Feng, Fouque \& Ichiba \cite{fengFouqueIchiba2020linearquadratic}, which mainly studies linear-quadratic stochastic differential games on deterministic directed chains. 

In stochastic differential games on directed networks, the state processes of all players are described by a stochastic differential system. Each player is interacting with other players through its cost function and the aim of every player is to minimize this cost function by controlling its state. Roughly speaking, the state process of one player is associated with a vertex of the network graph. When the graph is directed and if there is an arrow from $j$ to $i$, the cost function of player $i$ depends on the state process of player $j$. Furthermore, when the graph is random, the cost function of player $i$ depends on the state process of player $j$ with some probability of the presence of an arrow from $j$ to $i$.

The goal of studying the stochastic differential game problem on networks is to determine and analyze the Nash equilibria of the game for different types of networks. There are  two extreme types of networks describes as follows. 

On one hand, we can consider a fully connected network with interaction of mean-field type, described in Figure \ref{fig:chain}(a). When the number of players goes to infinity, with appropriate scaling, this kind of game can be approximated by a mean field game. This approximation problem by mean field games has been widely discussed, for instance in Lasry and Lions \cite{LASRY2006619, LASRY2006679, Lasry2007} and Lacker \cite{lacker2020}. Stochastic games on infinite random networks have been proposed and studied. For instance, Delarue \cite{delarue:hal-01457409} discussed a simple toy model with a large number of players in mean field interaction when the graph connection between them is not complete but is of Erd\H{o}s-R\'enyi type.
Recently, Caines and Huang \cite{CainesHuang8619367, CainesHuang9029871} investigated Graphon Mean Field Games which relate infinite population equilibria on infinite networks to finite population equilibria on finite networks.

On the other hand, the network can be very sparse, structured network. Detering, Fouque \& Ichiba \cite{Nils-JP-Ichiba2018DirectedChain} studied a particle system interacting through a one-dimensional directed chain structure without the game aspect. Then Feng, Fouque \& Ichiba \cite{fengFouqueIchiba2020linearquadratic} investigated linear-quadratic games on a finite, directed chain of $N$ vertices Figure \ref{fig:chain}(b), where there are arrows from $i+1$ to $i$ for $i = 1,...,N-1$ and a boundary condition at the vertex $N$. There are only $N-1$ directed edges in the network in contrast to the fully connected graph, where there are ${N\choose 2}$ undirected edges. It is a complete opposite situation to the mean field games since each player interacts only with its neighbor in a given direction on a directed chain network. 

The objective of our paper is to investigate linear-quadratic stochastic differential games on random directed networks and to find their open-loop Nash equilibria explicitly in a similar spirit of the work by Carmona, Fouque and Sun \cite{CarmonaFouqueSunSystemicRisk}. 
We propose first a stochastic game on a random directed chain network shown in Figure \ref{figure:randomchain}. Then, we generalize the result to  stochastic differential games on a random two-sided directed chain and  on a random directed tree structure as two extensions of random directed chain graphs. In this framework, the graph represents interactions among players through the cost functions but not necessarily reflects physical (spatial) distance among players. The notion of neighbor refers to the presence of a link (edge) in the graph.

The paper is organized as follows. In Section \ref{section random chain}, we study a stochastic game with infinite players on a directed chain structure and construct an open-loop Nash equilibrium of the system. We assume that the interaction between two neighbor is random but frozen in time and {\it i.i.d.} among all the successive pairs of neighbors. Section \ref{section 3} is devoted to the analysis of an extension of Section \ref{section random chain}, which considers a game for countably many players with random double-sided interactions and studies the effect of random double-sided interactions on the open-loop Nash equilibrium. We extend our results to a directed tree structure with random interactions between players in the neighboring generations in Section \ref{section-tree-model}. We conclude in Section \ref{conclusion} and Appendix \ref{Appendix} includes some technical proofs and discussions.

\begin{figure}[!h]
\centering
\begin{minipage}[t]{4.5cm}
\begin{tikzpicture}[scale=1.4, cap=round, >={triangle 60}]
\draw[] (0, 0) circle (1cm);
\foreach \x in {0,30,..., 360} {
	\filldraw[black] (\x: 1cm) circle(0.05cm);
	\foreach \y in {0,30,...,360}{
		\draw (\x: 1cm) -- (\y: 1cm); 
	}
}	
\draw (1.4,0) node[]{\small $3$};
\draw (1.21244,0.7) node[]{\small $2$};
\draw (0.7, 1.21244) node[]{\small $1$};
\draw (0, 1.4) node[]{\small $N$};
\draw (-0.5, 1.21244) node[left]{\small $N-1$};
\end{tikzpicture}
\subcaption{ }
\end{minipage}
\hspace{3cm}
\begin{minipage}[t]{4cm}
\begin{tikzpicture}
[scale=1,mydashed/.style={dashed},>={triangle 60}]
\draw[thick] (0,0.1) -- (0,0) node[label=above:1] {};
\draw[thick] (1,0.1) -- (1,0) node[label=below:2] {};
\draw[thick] (2,0.1) -- (2,0) node[label=above:3] {};
\draw[thick] (4,0.1) -- (4,0)--(5,0) -- (5,0.1);
\draw[thick,<-] (0,0) -- (1,0);
\draw[thick,<-] (1,0) -- (2,0);
\draw[thick,mydashed,<-] (2,0) -- (4,0) node[label=below:{\small $N-1$}] {};
\draw[thick,<-] (4,0) -- (5,0) node[label=above:{$N$}] {};
\end{tikzpicture}
\subcaption{ }
\end{minipage} 

\caption{(a) Fully connected graph, (b) Finite directed chain graph.}\label{fig:chain}
\end{figure}
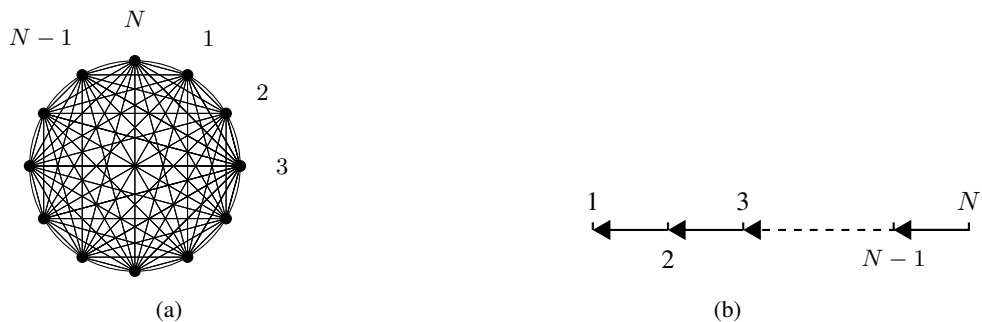

\newpage

\section{Random Directed Chain Game}\label{section random chain}
\subsection{Setup and Assumptions}
In Feng, Fouque \& Ichiba \cite{fengFouqueIchiba2020linearquadratic}, we have studied a stochastic game with infinite players on a directed chain structure and found an open-loop Nash equilibrium of the system. In this paper, we are still looking at an infinite-player system but assuming the interaction between every two neighbors is random as follows. We introduce a binary random variable $R_n$ which represents the random interaction between player $n$ and $n+1$. The $\{R_n,\,n\geq 1\}$ are independent and identically distributed random variables taking values in $\{0,1\}$ with probabilities $p_0$ and  $p_1=1-p_0$. 
When $R_n$ is zero, we assume player $n$ has no interaction with player $n+1$. An example of the chain structure is shown in Figure \ref{figure:randomchain}.

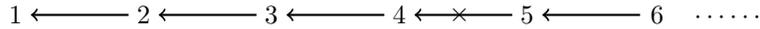
\begin{figure}[!h]
	\centering
		\begin{tikzpicture}[scale=1,mydashed/.style={dashed,dash phase=3pt}]
\draw(1.8,0) node{$1$};
\draw[thick, <-] (2,0) -- (3.3,0); 
\draw(3.5,0) node{$2$};
\draw[thick,<-] (3.7,0) -- (5,0);
\draw(5.2,0) node{$3$};
\draw[thick,<-] (5.4,0) -- (6.7,0);
\draw(6.9,0) node{$4$};
\draw[thick,<-] (7.1,0) -- (8.4,0);
\draw(7.7,0) node{$\times$};
\draw(8.6,0) node{$5$};
\draw[thick,<-] (8.8,0) -- (10.1,0);
\draw(11,0) node{$6\quad\cdots\cdots$};
\end{tikzpicture}
	\caption{Example of a Random Directed Chain: $R_1=R_2=R_3=R_5=1$; $R_4=0$}
	\label{figure:randomchain}
\end{figure}

We assume the dynamics of the states of all players are given by the stochastic differential equations of the form: for $i\geq 1$
 \begin{equation}\label{eq:1}
    dX_t^i=\a_t^i {\mathrm d}t+\s {\mathrm d}W_t^i,\quad 0\leq t\leq T,
\end{equation}
where $ (W_t^i)_{0\leq t\leq T},\, i \geq 1$ are one-dimensional independent standard Brownian motions. Here and throughout the paper, the argument in the superscript represents index or label but not the power. For simplicity, we assume that the diffusion is one-dimensional and the diffusion coefficients are constant and identical denoted by $\sigma>0$. The drift coefficients $\alpha^i$'s are adapted to the filtration of the Brownian motions and satisfy $\mathbbm{E}[\int_0^T |\alpha_t^i|^2 dt]<\infty$ for $i \geq 1$. The system starts at time $t = 0$ from $i.i.d.$ square-integrable random variables $X_0^i = \xi_i$, independent of the Brownian motions and, without loss of generality,  we assume ${\mathbbm E}(\xi_i) = 0$ for $i \geq 1$. 

In this model, each player $i$ chooses its own strategy $\alpha^i$, in order to minimize its objective function given by:
\begin{align*}
J^{i}(\boldsymbol{\alpha})=&\mathbbm{E}_{X,R} \bigg\{ \displaystyle \int_0^T \big(\frac{1}{2}(\a^{i}_t)^2
+\frac{\e}{2}(X_t^{i+R_i}-X_t^{i})^2\big) {\mathrm d}t+\frac{c}{2}(X_T^{i+R_i}-X_T^{i})^2\bigg\}\\
=&\mathbbm{E}_{X} \bigg\{ \displaystyle \int_0^T \big(\frac{1}{2}(\a^{i}_t)^2
+\frac{\e}{2}(X_t^{i+1}-X_t^i)^2\cdot p_1\big) {\mathrm d}t+\frac{c}{2}(X_T^{i+1}-X_T^i)^2\cdot p_1\bigg\},
\end{align*}
for some constants $\e>0$, $c\geq 0$ and $\boldsymbol{\alpha}=(\alpha^1,\alpha^2,\ldots)$ with $\alpha^i\in \mathbb{R}$. According to the objective function, if a player is not in interaction with its right neighbor, then we assume she has no incentive to do anything. This is a \emph{Linear-Quadratic} differential game on a directed chain network, since the state $X^{i}$ of each player $i$ interacts only with $X^{i+1}$ of player $i+1$ through the quadratic cost function for $i \geq 1$. 

\begin{remark}
When every player is connected with the next one, i.e. $p_1=1$, we get back to the stochastic game on a directed chain structure, studied in Feng, Fouque \& Ichiba \cite{fengFouqueIchiba2020linearquadratic}.
\end{remark}

\subsection{Open-Loop Nash Equilibrium} 
In this section, we search for an open-loop Nash equilibrium of the system among the admissible strategies $\{\alpha_t^i,i\geq 1, t \in [0, T] \}$. We construct the equilibrium by using the Pontryagin stochastic maximum principle (see \cite{carmona2013forwardbackward} for stochastic maximum principle in the context of mean-field games).

The corresponding Hamiltonian for player $i$ is given by:
\begin{align}\label{eq: Hamiltonian random game}
H^{i}(x^1,x^2,\cdots,y^{i,1},\cdots, y^{i,n_i},\a^1,\a^2,\cdots)=\sum_{k=1}^{n_i} \a^{k}y^{i,k}
+\frac{1}{2}(\a^{i})^2+
\frac{\e}{2}(x^{i+1}-x^{i})^2\cdot p_1,
\end{align}
assuming it is defined on real numbers $x^i, y^{i,k}, \alpha^i, i \geq 1, k \geq 1$, where only finitely many $y^{i,k}$ are non-zero for every
given $i$. Here, $n_i$ is a finite number depending on $i$ with $n_i > i$. This assumption is checked in Remark \ref{finite-check} below. Thus, the Hamiltonian $H^i$ is well defined for $i \geq 1$.

The value of $\hat\alpha^i$ minimizing the Hamiltonian $H^{i}$ with respect to $\alpha^i$, when all the other variables including $\alpha^j$ for $j\neq i$ are fixed, is given by the first order condition 
\[
    \partial_{\alpha^i}H^i=y^{i,i}+\alpha^i=0 \quad \text{leading to the choice:} \quad  \Hat{\alpha}^i=-y^{i,i}.
\]

The adjoint processes $Y_t^i=(Y_t^{i,j};j=1,\ldots,n_i)$ and $Z_t^i=(Z_t^{i,j,k};1\leq j\leq n_i,k\geq 1)$ for $i\geq 1$ are defined as the solutions of the system of  backward stochastic differential equations (BSDEs): for $i\geq 1$, $1\leq j\leq n_i$
\begin{equation} 
\left\{
  \begin{split}
    {\mathrm d}Y_t^{i,j}
    &=-\partial_{x^j}H^i(X_t,Y_t^i,\alpha_t){\mathrm d}t+\sum\limits_ {k=1}^\infty Z_t^{i,j,k}{\mathrm d}W_t^k=-p_1\cdot \epsilon (X_t^{i+1}-X_t^i)(\delta_{i+1,j}-\delta_{i,j}){\mathrm d}t+\sum\limits_ {k=1}^\infty Z_t^{i,j,k}{\mathrm d}W_t^k, \\
    Y_T^{i,j}&=\partial_{x^j}g_i(X_T)=p_1\cdot c(X_T^{i+1}-X_T^i)(\delta_{i+1,j}-\delta_{i,j});
      \end{split}
      \right.
\end{equation}
for $0 \leq t \leq T$. Particularly, for $j=i$ and $j = i+1$, it becomes:
\begin{equation}\label{BSDE-general}
\left\{
  \begin{split}
    & {\mathrm d} Y_t^{i,i} =p_1\cdot\epsilon \,(X_t^{i+1}- X_t^i)\,{\mathrm d}t+\sum\limits_ {k=1}^\infty Z_t^{i,i,k}{\mathrm d}W_t^k, \quad &Y_T^{i,i}=-p_1\cdot c \,(X_T^{i+1}-X_T^i) , \\
    & {\mathrm d} Y_{t}^{i,i+1} \, =\,  -p_1\cdot \epsilon (X_t^{i+1}-X_t^i){\mathrm d}t+\sum\limits_ {k=1}^\infty Z_t^{i,i+1,k}{\mathrm d}W_t^k,\quad &Y_T^{i,i+1}= p_1\cdot c(X_T^{i+1}-X_T^i) .
 \end{split}
\right.
\end{equation}

\begin{remark} \label{finite-check}
When $j\neq i, i+1$, ${\mathrm d}Y_{t}^{i,j} \, =\,  \sum\limits_{k=1}^\infty Z_{t}^{i,j,k} {\mathrm d} W_{t}^{k}$ and $ Y_{T}^{i,j} \, =\,  0$, which gives $Z_{t}^{i,j,k} \equiv 0 ,  \, 0 \le t \le T\,$ for all $k$. Thus $Y_{t}^{i,j} \equiv 0 \,,\,0 \le t \le T\,$ for all $j\neq i, i+1$. There must be finitely many non-zero $Y^{i,j}$'s for every $i$. Hence, the Hamiltonian $H^{i}$ in \eqref{eq: Hamiltonian random game} can be rewritten as
\begin{equation*}
    H^i(x^1,x^2,\cdots,y^{i,i},y^{i,i+1},\a^1,\a^2,\cdots)= \a^i y^{i,i}+\a^{i+1} y^{i,i+1}+\frac{1}{2}(\a^i)^2+\frac{\e}{2}(x^{i+1}-x^i)^2\cdot p_1.
\end{equation*} 
We also note that $Y_{t}^{i,i+1}=-Y_{t}^{i,i}$ and $Z_t^{i,i+1,k}=-Z_t^{i,i,k}$, so that it's enough to find $Y_{t}^{i,i}$.
\end{remark}

\bigskip
Considering the BSDE system and its terminal condition, we  make an ansatz of the form: 
\begin{equation}\label{ansatz-general}
    Y_t^{i,i}=\sum\limits_{j=i}^\infty \phi_t^{i,j}X_t^j, \quad 0\leq t\leq T
\end{equation}
for some deterministic scalar functions $\phi_t$ satisfying the terminal conditions: $\phi_T^{i,i}=p_1c, \,\phi_T^{i,i+1}=-p_1c,$ and $\phi_T^{i,j}=0$ for $j\geq i+2$.

Substituting the ansatz, the optimal strategy $ \hat{\alpha}^i$ and the controlled forward equation for $X^i$ in (\ref{eq:1}) become
\begin{equation} \label{eq: OL-Nash1}
    \left\{
  \begin{array}{ll}
    &\hat{\alpha}_t^i=-Y_t^{i,i}=-\sum\limits_{j=i}^\infty \phi_t^{i,j}X_t^j,\\
    & {\mathrm d}X_t^j=-\sum\limits_{k=j}^\infty \phi_t^{j,k}X_t^k {\mathrm d}t+\sigma {\mathrm d}W_t^j.
  \end{array}
\right.
\end{equation}
Differentiating the ansatz (\ref{ansatz-general}) and substituting (\ref{eq: OL-Nash1}) leads to:
\begin{equation}\label{ito-general}
  \begin{array}{ll}
    {\mathrm d}Y_t^{i,i}&=\sum\limits_{j=i}^\infty [X_t^j\Dot{\phi}_t^{i,j} {\mathrm d}t+\phi_t^{i,j}dX_t^j]\\
   &= \sum\limits_{k=i}^\infty \big(\Dot{\phi}_t^{i,k}- \sum\limits_{j=i}^k \phi_t^{i,j}\phi_t^{j,k} \big) X_t^k{\mathrm d}t
   +\sigma\sum\limits_{k=i}^\infty \phi_t^{i,k}{\mathrm d}W_t^k.
  \end{array}
\end{equation}
Here $\dot{\phi}_t$ represents the time derivative of $\phi_t$. Comparing the martingale terms and drift terms of the two It\^o's decompositions  (\ref{BSDE-general}) and (\ref{ito-general}) of $Y_t^{i,i}$, 
the martingale terms give the deterministic (and therefore adapted) processes 
$Z_t^{i,i,k}$:
\begin{equation}\label{random chain martingale}
    Z_t^{i,i,k}=0 \quad \text{ for } \quad k<i, \quad \text{ and }\ \quad Z_t^{i,i,k}=\sigma\phi_t^{i,k} \quad \text{ for } \quad k\geq i.
\end{equation}

Moreover, the drift terms show that the functions $\phi_t$ must satisfy the system of Riccati equations :\\
\begin{equation}\label{random chain riccati}
\begin{array}{lrll}
     &\Dot{\phi}_t^{i,i}&=\phi_t^{i,i}\cdot \phi_t^{i,i}-p_1\cdot \epsilon, &\phi_T^{i,i}=p_1\cdot c, \\
      &\Dot{\phi}_t^{i,i+1}&=\phi_t^{i,i}\cdot \phi_t^{i,i+1}+\phi_t^{i,i+1}\cdot \phi_t^{i+1,i+1}+p_1\cdot\epsilon, &\phi_T^{i,i+1}=-p_1\cdot c,\\
     \ell\geq i+2:\quad &\Dot{\phi}_t^{i,\ell}
     &=\sum\limits_{j=i}^l \phi_t^{i,j}\phi_t^{j,\ell} , & \phi_T^{i,\ell}=0,
\end{array}
\end{equation}

The Riccati system is solvable and the solutions only depend on the ``distance" $\ell-i$. Thus, if we define $\phi_t^{j-i}:=\phi_t^{i,j}$ for all $i\geq 1, j\geq i$ and $p:=p_1$, we can rewrite the system (\ref{random chain riccati})
\begin{equation}\label{random chain riccati-1}
\begin{array}{lrll}
     &\Dot{\phi}_t^{0}&=\phi_t^{0}\cdot \phi_t^{0}-p\epsilon, &\phi_T^{0}=pc, \\
      &\Dot{\phi}_t^{1}&=2\phi_t^{0}\cdot \phi_t^{1}+p\epsilon, &\phi_T^{i,i+1}=-p c,\\
     k\geq 2:\quad &\Dot{\phi}_t^{k}
     &=\sum\limits_{j=0}^k \phi_t^{j}\phi_t^{k-j} , & \phi_T^{k}=0.
\end{array}
\end{equation}

\begin{Lemma}\label{inf_sumo}
With  $c \geq 0 $, and $\varepsilon > 0 $, the solution to the Riccati system (\ref{random chain riccati-1})  satisfies 
\begin{align}
\sum\limits_{k=0}^{\infty} \,\phi_t^{k}=0;\quad\quad
\phi_{t}^{0}=\sqrt{p}\cdot\, \dfrac{(-\e-c\sqrt{p\e})e^{2\sqrt{p\e}(T-t)}+\e-c\sqrt{p\e}}{(-\sqrt{\e}-c\sqrt{p})e^{2\sqrt{p\e}(T-t)}-\sqrt{\e}+c\sqrt{p}} > 0\, \text{ when } p\neq 0,
\end{align}
for $0\le t \le T$. Moreover, the functions $\phi_t^k$'s are obtained by a series expansion  of the generating function $S_t(z) = \sum_{k=0}^{\infty} z^{k}\,\phi_t^{k}$, $ z \le 1$ of the sequence $\{\phi^{\ell}\}$ given by $S_{t} (1) \equiv 0 $, and if $z < 1 $, 
\begin{align}\label{S_t_random chain}
S_{t}(z) = \sqrt{p}\cdot \dfrac{\big(-\e(1-z)-c(1-z)\sqrt{p\e(1-z)}\big)\, e^{2\sqrt{p\e(1-z)}(T-t)}+\e(1-z)-c(1-z)\sqrt{p\e(1-z)}}{\big(-\sqrt{\e(1-z)}-\sqrt{p}c(1-z)\big)\, e^{2\sqrt{p\e(1-z)}(T-t)}-\sqrt{\e(1-z)}+\sqrt{p}c(1-z)}
\end{align}
for every $0 \le t \le T$. 
\end{Lemma}
\begin{proof} 
Define the generating function $S_t(z)=\sum_{k=0}^\infty z^k\ \phi_{t}^{(k)}$where $0\leq z< 1$ and $\phi_t^{(k)}=\phi_t^{k}$ in \eqref{random chain riccati-1} to avoid confusion with derivation. Then substituting \eqref{random chain riccati-1}, we obtain
\begin{equation}\label{Sol S_t}
\begin{split} 
\Dot{S}_t(z) &= \sum\limits_{k=0}^\infty {z}^k\Dot{\phi}_{t}^{(k)}=(S_t(z))^2-p\epsilon(1-z), \quad  0 \le t \le T ;\quad \quad 
S_T(z)=pc(1-z).
\end{split} 
\end{equation}

\noindent $\,\bullet\,$ For $z=1$, we get the ODE:
$\Dot{S}_t(1)=(S_t(1))^2\, ,\,S_T(1)=0.$
The solution is $S_t(1)\equiv 0$ for all $t$. Because the series defining $S_t(1)$ may not converge, we take a sequence $\{z_n\}\to 1$. The limit of $S_t(z_n)$ converges to the ODE above, and we can get the conclusion. Then we deduce:
\begin{equation*}
\sum\limits_{k=0}^\infty  \phi_{t}^{(k)}=0,\quad i.e.,\quad \phi_{t}^{(0)}=-\sum\limits_{k=1}^\infty \phi_{t}^{(k)}.
\end{equation*}
\noindent $\,\bullet\,$ For $z\neq 1$, the solution to the Riccati equation (\ref{S_t_random chain}) satisfies: 
\begin{equation*}
     \begin{split} 
     S_t(z)
     &=\dfrac{\big(-p\e(1-z)-pc\sqrt{p\e(1-z)}(1-z)\big)e^{2\sqrt{p\e(1-z)}(T-t)}+p\e(1-z)-pc\sqrt{p\e(1-z)}(1-z)}{\big(-\sqrt{p\e(1-z)}-pc(1-z)\big)e^{2\sqrt{p\e(1-z)}(T-t)}-\sqrt{p\e(1-z)}+pc(1-z)}\\
     &\xrightarrow[T\to \infty]{} 
     \sqrt{p\e(1-z)},
     \end{split} 
\end{equation*}
which gives \eqref{S_t_random chain}.
 \end{proof}
\begin{remark}
It follows from Lemma \ref{inf_sumo} that the forward dynamics (\ref{eq: OL-Nash1}) can be formally written as:
\begin{equation} \label{eq: OLNashInf1}
\begin{split}
      {\mathrm d}X_t^i&=- \sum_{j=0}^\infty\phi_t^{j}X_t^{i+j} {\mathrm d}t+\s {\mathrm d}W_t^i 
     =\phi_t^{0}\cdot \Big( \sum_{j=1}^\infty \frac{-\phi_t^{j}}{\phi_t^{0}}X_t^{i+j} -X_t^{i}\Big) {\mathrm d}t +\s {\mathrm d}W_t^i 
 \end{split} 
\end{equation}
for $i \ge 1$, $0 \le t \le T$. This is a mean-reverting type process, since $\phi_t^0>0$.
We also see that this system is invariant under the shift of indices of individuals. In particular, the law of $X^i$ is the same as the law of $X^1$ for every $i$. 
\end{remark} 
 
Here is a a summary of this section on the random infinite player game.
\begin{prop}
An open-loop Nash equilibrium for the random infinite-player game with cost functionals $J^i$ is determined by (\ref{eq: OLNashInf1}), where $\{ \phi_t^{j},\,0\leq t\leq T;\, j \ge 0 \} $ are the unique solution to the infinite system  (\ref{random chain riccati-1}) of Riccati equations.  
\end{prop}

\subsection{Stationary Solution and Catalan Markov Chain}
By taking $T \to \infty$, we look at the stationary solution of the Riccati system (\ref{random chain riccati-1}) satisfying $\Dot{\phi}_\cdot^{j}=0$ for all $j$. Without loss of generality, we assume $\epsilon=1$. Otherwise the solution should be multiplied by $\sqrt{\epsilon}$ for all $\{\phi^k,\,k\geq 1\}$.
Then the system gives the solutions and the recurrence relation: 
\[
\phi^0=\sqrt{p}, \quad \phi^1=-\frac{\sqrt{p}}{2}, \quad  \text{ and } \quad \sum\limits_{k=0}^n \phi^k\phi^{n-k}=0. 
\]
This is closely related to the recurrence relation of {\it Catalan} numbers. By using a moment generating function method as in Appendix \ref{moment_generating_method}, we obtain the stationary solution:
\[
\phi^0=\sqrt{p},\quad \phi^1=-\frac{\sqrt{p}}{2} ,\quad  \text{ and } \quad \phi^k=-\frac{(2k-3)!}{(k-2)!\,k!\,2^{2k-2}}\, \sqrt{p}  \quad \text{ for } \quad k\geq 2. 
\] 

Let $q_0=-\dfrac{\phi^0}{\sqrt{p}}=-1, q_1=-\dfrac{\phi^1}{\sqrt{p}}=\dfrac{1}{2} $,  and $q_k=-\frac{\phi^k}{\sqrt{p}}=\dfrac{(2k-3)!}{(k-2)!k!}\,\dfrac{1}{2^{2k-2}}\,$ for $k\geq 2$. By lemma \ref{inf_sumo}, we have the relation: $\sum\limits_{k=0}^\infty q_k=0$. Then we consider the continuous-time Markov chain $M(\cdot)$ with state space $\, \mathbb N \,$ and Catalan generator matrix 
\begin{equation} \label{eq: Q}
\mathbf{Q} \, =\left( \begin{array}{ccccc} 
q_0 & q_1 &  q_2 & q_3 & \cdots \\
0 & q_0 & q_1 &  q_2 & \ddots \\
0 &0 & q_0 & q_1 &\ddots\\
& \ddots & \ddots &\ddots  & \ddots\\
\end{array} 
\right) .
\end{equation} 
Note that the transition probabilities of the continuous-time Markov chain $M(\cdot)$, called a Catalan Markov chain,  are $p_{i,j}(t)=\, \mathbb P (M(t)=j|M(0)=i)=(e^{t\mathbf{Q}})_{i,j},\, i,j \ge 1, \, t\geq 0$. With replacement of $\phi^{j}_{t}$, $t\ge 0 $ by the stationary solution $\phi^{j}$ in (\ref{eq:1}) and assuming $\sigma=1$, the infinite particle system $ (X_{\cdot}^{i}$, $i \ge 1 )$ can be represented formally as a linear stochastic evolution equation:
\begin{equation}\label{eq: evolution_random_directed} 
\begin{split}
{\mathrm d}\mathbf{X}_t&= \,\sqrt{p}\,\mathbf{Q\,X}_t {\mathrm d}t+ {\mathrm d}\mathbf{W}_t; \quad t \ge 0 , 
\end{split}
\end{equation}
where $\mathbf{X_.}=(X_{.}^{k}, k \ge 1 )$ with $\mathbf{X}_0=\mathbf{x}_0$ and $\mathbf{W_.}=(W_{.}^{k} , k \ge 1 )$. 
Its solution is:
\begin{equation*}
\mathbf{X}_t=e^{t\sqrt{p}\,\mathbf{Q}}\mathbf{x_0}+ \int_0^t e^{(t-s)\sqrt{p}\,\mathbf{Q}}{\mathrm d}\mathbf{W}_s ;\quad t\geq 0.
\end{equation*}

Without loss of generality, let us assume $\mathbf{X}_0=\mathbf{0}$. Then, 
\begin{equation*}
\begin{split} 
{X}_t^{i}&= \int^{t}_{0}   \sum\limits_{j=i}^{\infty}(\exp (  (t-s)\sqrt{p}\, \mathbf{Q}))_{i,j} {\mathrm d} {W}_{s}^{j} \, =\mathop{\mathlarger{\int^t_0}}\,  \sum_{j=i}^\infty p_{i,j}(t-s) {\mathrm d} W_{s}^{j} 
 =\mathop{\mathlarger{\int^t_0}}\, \sum_{j=i}^\infty\, \mathbb P (M(t-s)=j|M(0)=i){\mathrm d} W_{s}^{j}\\
 &=\mathbbm{E}^M \Big[ \mathop{\mathlarger{\int^t_0}}\,\sum_{j=i}^\infty\, \mathbf{1}_ {(M(t-s)=j)} {\mathrm d} W_{s}^{j} | M(0)=0 \Big];\quad t\geq 0,
\end{split} 
\end{equation*}
where the expectation is taken with respect to the probability induced by the Markov chain $M(\cdot)$, independent of the Brownian motions $(W_{\cdot}^{j},\, j\in \mathbb{N}_0)$. This is a Feynman--Kac representation formula for the infinite particle system $\mathbf X_{\cdot}$ in (\ref{eq: evolution_random_directed}) associated with the continuous-time Markov chain $M(\cdot)$. We can compute explicitly the corresponding transition probability $(p_{i,j}(\cdot))$. \\

\begin{prop}\label{sol_markovc_random}
The Gaussian process $\, X_{t}^{i} \,$, $\, i \ge 1 \,$, $\, t \ge 0 \,$ in \eqref{eq: evolution_random_directed}, corresponding to the Catalan Markov chain, is given by
\begin{equation} \label{solution_mc_random}
\begin{split}
{X}_{t}^{i}  \, &=\, \sum\limits_{j=i}^{\infty}\int^{t}_{0}   (\exp ( \sqrt{p}\,Q (t-s)))_{i,j} {\mathrm d} {W}_{s}^{j} \, =\, \sum\limits_{j=i}^{\infty} \int^{t}_{0}  \frac{\,p^{j-i}\,(t-s)^{2(j-i)}  \,}{\, (j-i)!\,} \cdot F^{(j-i)}(-p(t-s)^{2})  {\mathrm d} W_{s}^{j}  \\
&\, =\,  \sum\limits_{j=i}^{\infty} \int^{t}_{0}  \frac{\,p^{j-i}\,(t-s)^{2(j-i)}  \,}{\, (j-i)!\,} \cdot \rho_{j-i}(- p(t-s)^{2}) \, e^{-\sqrt{p}\,(t-s)} \cdot  {\mathrm d} W_{s}^{j},
\end{split}
\end{equation}
where $\, {W}_{\cdot}^{j}  \,$, $ j \in \mathbb N \,$ are independent standard Brownian motions and $\rho_{i}(\cdot)$ is defined by  
\begin{equation} \label{eq: rhokxProp1_random}
\rho_{i}(x) =\frac{1}{2^i} \sum\limits_{j=i}^{2i-1}\, \frac{(i-1)!}{(2j-2i)!!(2i-j-1)!} \cdot (-x)^{\,-\frac{j}{2}}, 
\end{equation} 
for $i\geq 1$, and $\rho_{0}(x) \, =\,  1$ for $x \le 0 $.  
Moreover, when $p\neq 0$, the asymptotic variance of $X_t^{i}$, $i\geq 1$ is finite, i.e.   
\[
\lim\limits_{t \to \infty}\text{Var} ( X_{t}^{i})\,=\,\lim\limits_{t \to \infty}\text{Var} ( X_{t}^{1})\,=\,\dfrac{1}{\sqrt{2\,p}}.
\]
\end{prop}
\begin{proof}
Given in \textbf{Appendix} \ref{appendix_CatalanMarkovChainRandom}.
\end{proof}

\section{Random Two-sided Directed Chain Game} \label{section 3}
To extend the investigation of random directed chain in Section \ref{section random chain}, we will consider a linear-quadratic stochastic game for countably many players with random double-sided interactions over a finite time horizon $\,[0, T]\,$. We shall study the effect of random double-sided interactions on the open-loop Nash equilibrium and compare it with the random one-sided (directed) chain interaction in Section \ref{section random chain}. To represent the random interactions of player $n$ in two directions, we introduce the binary random variables $R_n$ and $L_n$. The $R_n$'s for $n\in \mathbb{Z}$ are independent and identically distributed random variables taking values in $\{0,1\}$ with  probabilities $p_0$ and  $p_1=1-p_0$. The $L_n$'s for $n\in\mathbb{Z}$ are also independent and identically distributed random variables taking values in $\{0,1\}$  with  probabilities $q_0=1-q_1$ and $q_{1}$. $\{R_n,n\in \mathbb{Z}\}$ is independent of $\{L_n,n\in \mathbb{Z}\}$.
When $R_n$ is one, we assume player $n$ is interacting with player $n+1$. When $L_n$ is one, we assume player $n$ is interacting with player $n-1$. The random variable $R_n$ affects the left arrow on the right of site $n$ and the random variable $L_n$ affects the right arrow  on the left of site $n$. Examples of the chain structure are shown in Figure \ref{figure:random2sidedchain}.

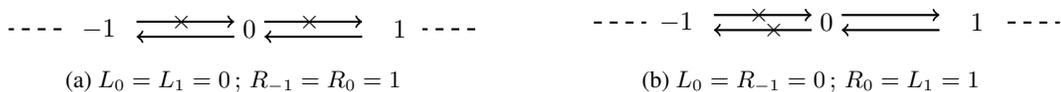
\begin{figure}[!h]
	\centering
	\begin{minipage}[t]{6cm}
		\centering
		\begin{tikzpicture}[scale=1,mydashed/.style={dashed,dash phase=1pt}]
\draw[thick, mydashed] (0.3,0) -- (1.1,0);
\draw(1.5,0) node{$-1$};
\draw[thick, <-] (2,-0.1) -- (3.3,-0.1);
\draw[thick, ->] (2,0.1) -- (3.3,0.1);
\draw(2.6,0.1) node{$\times$};
\draw(3.5,0) node{$0$};
\draw[thick,<-] (3.7,-0.1) -- (5,-0.1);
\draw[thick,->] (3.7,0.1) -- (5,0.1);
\draw(4.3,0.1) node{$\times$};
\draw(5.5,0) node{$1$};
\draw[thick, mydashed] (5.8,0) -- (6.6,0);
\end{tikzpicture}
\subcaption{$L_0=L_1=0\,;\, R_{-1}=R_0=1$}
	\end{minipage}
	\hspace{1.5cm}
	\begin{minipage}[t]{6cm}
		\centering
		\begin{tikzpicture}[mydashed/.style={dashed,dash phase=3pt}]
\draw[thick, mydashed] (0.3,0) -- (1.1,0);
\draw(1.5,0) node{$-1$};
\draw[thick, <-] (2,-0.1) -- (3.3,-0.1);
\draw(2.8,-0.1) node{$\times$};
\draw[thick, ->] (2,0.1) -- (3.3,0.1);
\draw(2.6,0.1) node{$\times$};
\draw(3.5,0) node{$0$};
\draw[thick,<-] (3.7,-0.1) -- (5,-0.1);
\draw[thick,->] (3.7,0.1) -- (5,0.1);
\draw(5.5,0) node{$1$};
\draw[thick, mydashed] (5.8,0) -- (6.6,0);
\end{tikzpicture}
\subcaption{$L_0=R_{-1}=0\,;\, R_{0}=L_{1}=1$}
	\end{minipage}
	\caption{Examples of Two-sided Directed Chain}
	\label{figure:random2sidedchain}
\end{figure}

We assume the dynamics of the states of all players are given by the one-dimensional stochastic differential equations of the form: for $i\in \mathbb{Z}$
 \begin{equation}\label{eq:2sidedSDE}
    dX_t^i=\a_t^i {\mathrm d}t+\s {\mathrm d}W_t^i,\quad 0\leq t\leq T,
\end{equation}
where $\,(W^{i})_{0\leq t \leq T}\,$, $\,i \in \mathbb Z\,$ are independent, standard Brownian motions, independent of the initial values $\,X_{0}^{i} \, :=\, \xi^{i}\,$, $\,i \in \mathbb Z \,$, the initial values $\,\xi^{i}\,$ are i.i.d. with finite second moments for $\,i \in \mathbb Z\,$, a positive constant $\, \sigma > 0 \,$ is fixed and $\,\alpha^{i}_{\cdot}\,$ is a control of player $\,i\,$ adapted to the filtration of the Brownian motions with $\,\mathbb E [ \int^{T}_{0} \lvert \alpha_{t}^{i}\rvert^{2} {\mathrm d} t ] < \infty\,$ for $\,i \in \mathbb Z\,$. 

In order to take into account the two-sided feature of the model, we introduce the parameter $p\,\in\,(0,1)$, which will measure the strength of the asymmetry between the right and left interactions. Notice that if $p=0$ or $1$, the chain is one-sided as already treated in Section \ref{section random chain}. We shall see how this parameter $p$ and the weighted average $\,p p_{1} + (1-p) q_{1}\,$ appear in the Nash equilibrium.  

In the model, player $\,i\,$ controls its own strategy $\,\alpha^{i}\,$ in order to minimizes the objective function defined by 
\begin{align}
J^{i}(\boldsymbol{\alpha})&=\mathbbm{E}_{X,L,R} \bigg\{ \displaystyle \int_0^T \big( \frac{1}{2}(\a^{i}_t)^2
+\frac{\varepsilon}{2}p\,(X_t^{i+R_i}-X_t^{i})^2+\frac{\varepsilon}{2}(1-p)\,(X_t^{i}-X_t^{i-L_i})^2\big) {\mathrm d}t\nonumber\\
&\quad\quad\quad\quad\quad+\frac{c}{2}p\,(X_T^{i+R_i}-X_T^{i})^2+\frac{c}{2}(1-p)\,(X_T^{i}-X_T^{i-L_i})^2\bigg\}  \label{eq: cost}\\
&=\mathbbm{E}_X \bigg\{ \displaystyle \int_0^T \big(\frac{1}{2}(\a^{i}_t)^2
+\frac{\varepsilon}{2}p\, \cdot\,p_1(X_t^{i+1}-X_t^{i})^2+\frac{\varepsilon}{2}(1-p)\, \cdot\,q_1 (X_t^{i}-X_t^{i-1})^2 \big) {\mathrm d}t\nonumber\\
&\quad\quad\quad\quad\quad+\frac{c}{2}p\,\cdot\,p_1(X_T^{i+1}-X_T^{i})^2+\frac{c}{2}(1-p)\,\cdot\,q_1(X_T^{i}-X_T^{i-1})^2\bigg\},\nonumber\\
&:= \mathbbm{E}_X\Big[ \int^{T}_{0}  f^{i}(X_{t}, \alpha_{t}^{i})  {\mathrm d} t + g^{i}(X_{T}) \Big] \,  ,\nonumber\\
\text{where }&
f^{i}(x, \alpha^{i}) \,:=\, \frac{\,1\,}{\,2\,}(\alpha^{i})^2 + \frac{\varepsilon}{2} p\, \cdot\,p_1(x^{i+1} - x^{i})^{2}+  \frac{\varepsilon}{2} (1-p)\, \cdot\,q_1(x^{i} - x^{i-1})^{2} \, , \nonumber\\
&g^{i}(x) \, :=\,   \frac{c}{2} p\, \cdot\,p_1(x^{i+1} - x^{i})^{2}+  \frac{c}{2} (1-p)\,\cdot\,q_1 (x^{i} - x^{i-1})^{2} ,\nonumber
\end{align}

for some constants $\varepsilon >0$, $c\geq 0$, and $\, x = (x^{i}, i \in \mathbb Z ) \,$, $\alpha=(\alpha^i: i \in \mathbb Z)$ with $\alpha^{i} \in \mathbb R\,$. Each player optimizes the cost determined by the mixture of two criteria: distance from the right neighbor in the directed chain with weight $p$ and distance from the left neighbor with weight $1 - p$. Here, the superscript $\,i\,$ indicates the index but not the power. The functions $\, f^{i}\,$ and $\, g^{i}\,$ denote the running cost and terminal cost  of player $\,i\,$, respectively. 
To simplify some notations, let us write $\, \mathcal S \, :=\, \mathbb R^{\mathbb Z}\,$ and $\,\mathcal S^{2} \, :=\, \mathbb R^{\mathbb Z \times \mathbb Z}\,$. 

\subsection{Open-Loop Nash Equilibrium}
We search for Nash equilibrium of the system among strategies $\{\alpha^i,\, i\in \mathbb{Z}\}$. We construct an open-loop Nash equilibrium by the Pontryagin stochastic maximum principle.
The corresponding Hamiltonian for player $i$ is defined by
\begin{align} 
H^{i}(x, y, \alpha) \, &:=\, \sum_{k=-\infty}^{\infty} \alpha^{k} \, y^{i, k} + f^{i}(x, \alpha^{i}) \, ; \label{eq: Hamiltonian2sided}\\
&= \, \sum_{k=-n_i}^{n_i} \alpha^{k} \, y^{i, k} + \frac{\,1\,}{\,2\,}(\alpha^{i})^2 + \frac{\varepsilon}{2} p\, \cdot\,p_1(x^{i+1} - x^{i})^{2}+  \frac{\varepsilon}{2} (1-p)\,\cdot\,q_1 (x^{i} - x^{i-1})^{2} \, , \nonumber
\end{align}
for $\, x, \alpha \in \mathcal S\,$, $\, y \in \mathcal S^{2}\,$, $\,i\in \mathbb Z\,$, where only finitely many $y^{i,k}$ are non-zero for every
given $i$. Here, $n_i$ is a finite positive number depending on $i$ with $n_i > |i|$. This assumption is checked in Remark \ref{rmk:finiteYcheck2sided} below. Thus, the Hamiltonian $H^i$ is well defined for every $i$.

The value of $\hat\alpha^i$ minimizing the Hamiltonian $H^{i}$ with respect to $\alpha^i$, when all the other variables including $\alpha^j$ for $j\neq i$ are fixed, is given by the first order condition 
\[
    \partial_{\alpha^i}H^i=y^{i,i}+\alpha^i=0 \quad \text{leading to the choice:} \quad  \Hat{\alpha}^i=-y^{i,i}.
\]

The adjoint processes $Y_t^i=(Y_t^{i,j};-n_i\leq j\leq n_i)$ and $Z_t^i=(Z_t^{i,j,k};-n_i\leq j\leq n_i,k\in\mathbb{Z})$ for $i\in\mathbb{Z}$ are defined as the solutions of the system of  backward stochastic differential equations (BSDEs): for $i\in\mathbb{Z}$, $-n_i\leq j\leq n_i$
\begin{equation} \label{BSDE2sided}
\left\{
  \begin{array}{ll}
    {\mathrm d}Y_t^{i,j}
    &=-\partial_{x^j}H^i(X_t,Y_t^i,\alpha_t){\mathrm d}t+\sum\limits_ {k=-\infty}^\infty Z_t^{i,j,k}{\mathrm d}W_t^k\\
    &=-\bigg( \varepsilon p\,p_1 (X_t^{i+1}-X_t^i)(\delta_{j,i+1}-\delta_{j,i})+ \varepsilon (1-p)\,q_1 (X_t^{i}-X_t^{i-1})(\delta_{j,i}-\delta_{j,i-1}) \bigg){\mathrm d}t+\sum\limits_ {k=-\infty}^\infty Z_t^{i,j,k}{\mathrm d}W_t^k, \\
    Y_T^{i,j}&=\partial_{x^j}g_i(X_T)\,=\,c p\,p_1 (X_T^{i+1}-X_T^i)(\delta_{j,i+1}-\delta_{j,i}) + c (1-p)\,q_1 (X_T^{i}-X_T^{i-1})(\delta_{j,i}-\delta_{j,i-1});
      \end{array}
      \right.
\end{equation}
for $0 \leq t \leq T$. Particularly, for $j=i$, it becomes:
\begin{equation}\label{2sided BSDE i=j}
\left\{
  \begin{array}{ll}
     {\mathrm d} Y_t^{i,i} &=[\varepsilon p\,p_1(X_t^{i+1}-X_t^i)-\varepsilon (1-p)\,q_1(X_t^{i}-X_t^{i-1})]\,{\mathrm d}t+\sum\limits_ {k=-\infty}^\infty Z_t^{i,i,k}{\mathrm d}W_t^k\\
    &=[-\varepsilon \,\big(pp_1+(1-p)q_1\big)X_t^i+\varepsilon p p_1\, X_t^{i+1}+\varepsilon (1-p)q_1\, X_t^{i-1}]\,{\mathrm d}t+\sum\limits_ {k=-\infty}^\infty Z_t^{i,i,k}{\mathrm d}W_t^k,\\
    Y_T^{i,i}&=c\,\big(pp_1+(1-p)q_1\big) X_T^{i}-cpp_1\, X_T^{i+1}-c(1-p)q_1 \,X_T^{i-1}, 
 \end{array}
\right.
\end{equation}

\begin{claim} \label{eq: YZ identities}
In the case of a deterministic  two-sided directed chain, i.e. $p_0=q_0=0$, $p_1=q_1=1$ and $0<p<1$, we have for $\,i \in \mathbb Z\,$, 
\begin{equation}   \label{eq: YZ identities}
Y_{\cdot}^{i,i-1} + Y_{\cdot}^{i,i} + Y_{\cdot}^{i,i+1} \equiv 0   \, , \quad Z_{\cdot}^{i,i-1} + Z_{\cdot}^{i,i} + Z_{\cdot}^{i,i+1} \equiv 0  \, .
\end{equation}
This  is quite different from the one-sided directed chain case where the effect of player $\,i-1\,$ does not appear. 
\end{claim}
\begin{proof}
First, for the relation among player $\,i\,$ and players $\,i\pm 1\,$, note from \eqref{eq: Hamiltonian2sided} that   for each $\,i \in \mathbb Z\,$,
\begin{align*}
&\partial_{i} H^{i} \,: =\, \frac{\,\partial H^{i}\,}{\,\partial x^{i}\,} (x, y, \alpha)  \,= \, - \varepsilon p ( x^{i+1}- x^{i})  + \varepsilon (1-p)( x^{i} - x^{i-1}),\\
&\partial_{i+1} H^{i} \,: =\,\frac{\,\partial H^{i}\,}{\,\partial x^{i+1}\,} (x, y, \alpha) \,=\, \varepsilon p( x^{i+1}- x^{i}),\quad \partial_{i-1} H^{i} \,: =\, \frac{\,\partial H^{i}\,}{\,\partial x^{i-1}\,} (x, y, \alpha)\,= \, -\varepsilon (1-p) ( x^{i}- x^{i-1}), 
\end{align*}
and hence, 
\begin{equation*}
\partial_{i} H^{i} \, =\,  - (\partial _{i+1} H^{i} + \partial_{i-1} H^{i}) \,, \quad \text{ and } \quad  \partial_{i} g^{i} \, =\,  - (\partial _{i+1} g^{i} + \partial_{i-1} g^{i}) \, . 
\end{equation*}
Thus, (according to (\ref{BSDE2sided})), we claim that for $\,i \in \mathbb Z\,$, 
\begin{equation*}
Y_{\cdot}^{i,i-1} + Y_{\cdot}^{i,i} + Y_{\cdot}^{i,i+1} \equiv 0   \, , \quad Z_{\cdot}^{i,i-1} + Z_{\cdot}^{i,i} + Z_{\cdot}^{i,i+1} \equiv 0  \, , 
\end{equation*}
\end{proof}
\begin{remark} \label{rmk:finiteYcheck2sided}
We can also see from (\ref{BSDE2sided}) that $\, Y_{\cdot}^{i,j} \equiv 0 \,$, $\, Z_{\cdot}^{i,j,k} \equiv 0 \,$ whenever $\,j \neq i-1, i, i+1\,$. Thus there must be finitely many non-zero $Y^{i,j}$'s for every $i$.
\end{remark}

For each $\,i \in \mathbb Z\,$, we make the {\it ansatz } 
\begin{equation} \label{eq: ansatz}
Y_{t}^{i,i} \, =\, \sum_{k=-\infty}^{\infty} \phi_{t}^{i,k} X_{t}^{k} + \psi_{t}^{i} \, ; \quad i \in \mathbb Z \, , \, \, 0 \le t \le T \,, 
\end{equation}
where $\, (\phi_{\cdot}^{i,j}, i, j \in \mathbb Z )\,$, $\,(\psi^{i}_{\cdot}, i \in \mathbb Z ) \,$ are some differentiable deterministic functions satisfying terminal conditions: $\phi_{T}^{i,i}=c\,\big(pp_1+(1-p)q_1\big),\,\phi_{T}^{i,i+1}=-cpp_1,\,\phi_{T}^{i,i-1}=-c(1-p)q_1,\,\phi_{T}^{i,k}=0\,$otherwise and $\psi_{T}^{i}=0$ for $i\in \mathbb{Z}$; and
$\,\phi^{i,k}_{\cdot}\,$ is assumed to be shift invariant, that is, it depends only on the difference $\, k-i  \,$ but not on the values $i,k$ themselves. Substituting the {\it ansatz}, the optimal strategy $\hat{\alpha}^i$ and the forward equation for $X^i$ in \eqref{eq:2sidedSDE} become
\begin{equation} \label{eq: ControlandForwardEq}
    \left\{
  \begin{array}{ll}
    &\hat{\alpha}_t^i=-Y_t^{i,i}=-\, \sum\limits_{k=-\infty}^{\infty} \phi_{t}^{i,k} X_{t}^{k} - \psi_{t}^{i},\\
    & {\mathrm d}X_t^j=(-\, \sum\limits_{k=-\infty}^{\infty} \phi_{t}^{i,k} X_{t}^{k} - \psi_{t}^{i}) {\mathrm d}t+\sigma {\mathrm d}W_t^j.
  \end{array}
\right.
\end{equation}
Using the ``dot" notation for derivatives with respect to $\,t\,$ and differentiating the ansatz (\ref{eq: ansatz}) and substituting (\ref{eq: ControlandForwardEq}) leads to:
\begin{equation} \label{ito_2}
\begin{split}
{\mathrm d} Y_{t}^{i,i} \, &=\, \sum_{k=-\infty}^{\infty} \phi_{t}^{i,k} {\mathrm d} X_{t}^{k} + ( \dot{\psi}_{t}^{i} + \sum_{j=-\infty}^{\infty} \dot{\phi}^{i,j}_{t} X_{t}^{j} ) {\mathrm d}t \\
\, &=\, \Big [ \sum_{\ell=-\infty}^{\infty} ( - \sum_{k=-\infty}^{\infty}  \phi_{t}^{i,k} \phi_{t}^{k, \ell} + \dot{\phi}^{i,\ell}_{t} ) X_{t}^{\ell} + \dot{\psi}_{t}^{i} - \sum_{k=-\infty}^{\infty} \phi_{t}^{i,k} \psi_{t}^{k} \Big] {\mathrm d} t + \sigma \sum_{k=-\infty}^{\infty} \phi_{t}^{i,k} {\mathrm d} W_{t}^{k}.
\end{split}
\end{equation}
Comparing the finite variation and local martingale parts of the semimartingale decompositions ((\ref{2sided BSDE i=j}) and (\ref{ito_2})), we derive 
\begin{equation}
Z^{i,i,k}_{t}\equiv \sigma \phi^{i,k}_{t}   \, ; \quad 0 \le t \le T \, , \, i\in \mathbb{Z};
\end{equation}
and the following system of ordinary differential equations of Riccati type: 
\begin{equation} \label{eq: ODEs} 
\begin{split}
\dot{\psi}_{t}^{i} \, & =\,  \sum_{k=-\infty}^{\infty} \phi_{t}^{i,k} \psi_{t}^{k} \, , \\
\dot{\phi}_{t}^{i,j} \, &=\,  \sum_{k=-\infty}^{\infty} \phi_{t}^{i,k} \phi_{t}^{k, j} + \delta_{j, i+1} \cdot \varepsilon p\, p_1 -  \delta_{j,i}  \cdot \varepsilon \,\big(pp_1+(1-p)q_1\big)+ \delta_{j,i-1} \cdot   \varepsilon (1-p)\, q_1
\end{split}
\end{equation}
for $\, 0 \le t \le T \,$, $\, i,j \in \mathbb Z \,$ with the terminal conditions 
\begin{equation}  \label{eq: terminal conditions} 
\phi_{T}^{i,i} \, =\,  c\,\big(pp_1+(1-p)q_1\big) \,, \quad \phi_{T}^{i,i+1} \, =\, - c\,pp_1 \, , \quad \phi_{T}^{i,i-1} \, =\, - c\,(1-p)q_1 , \quad \phi^{i,j}_{T} \equiv 0 \,, \, \,  j \neq i-1, i, i+1 \, , 
\end{equation}
and $\,\psi_{T}^{i} \equiv 0 \,$ for $\, i \in \mathbb Z \,$.

\subsection{Discussion of the Riccati System \eqref{eq: ODEs}}
Since we make the ansatz \eqref{eq: ansatz} shift invariant, that is, $\,\phi^{i,k}_{\cdot} \, $ depends only on the difference $\,k-i\,$, we may write $ \phi^{i,k}_{\cdot}\, =\, \varphi^{k-i}_{\cdot} \,$ for some function $\,\varphi_{t}^{j}\,$, $\,j \in \mathbb Z\,$, $\, 0 \le t \le T\,$. Here,  note that the superscript $\,j\,$ is the index but not the power. Then we may rewrite \eqref{eq: ODEs} for $\,\phi^{i,k}_{\cdot }\,$ as the following ordinary differential equation for $\,\varphi_{\cdot}^{j}\,$, $\, j \in \mathbb Z\,$: 
\begin{equation}
\dot{\varphi}^{j}_{t} \, =\, \sum_{k=-\infty}^{\infty} \varphi_{t}^{k} \cdot \varphi_{t}^{j-k} + \delta_{j,1} \cdot \varepsilon p\,p_1 -  \delta_{j,0} \cdot \varepsilon \,\big(pp_1+(1-p)q_1\big)+ \delta_{j,-1} \cdot \varepsilon (1-p)\,q_1 \, ; \quad j \in \mathbb Z \, , \, \, 0 \le t \le T \, ,
\end{equation}
i.e., 
\begin{equation} \label{reduced ODE}
\left\{
\begin{array}{ll}
&\dot{\varphi}^{0}_{t} \, =\, \sum\limits_{k=-\infty}^{\infty} \varphi_{t}^{k} \cdot \varphi_{t}^{-k} - \varepsilon\,\big(pp_1+(1-p)q_1\big) ,\\
&\dot{\varphi}^{1}_{t} \, =\, \sum\limits_{k=-\infty}^{\infty} \varphi_{t}^{k} \cdot \varphi_{t}^{1-k}+\varepsilon p\,p_1, \\
&\dot{\varphi}^{-1}_{t} \, =\, \sum\limits_{k=-\infty}^{\infty} \varphi_{t}^{k} \cdot \varphi_{t}^{-1-k} +\varepsilon (1-p)\,q_1,\\
&\dot{\varphi}^{j}_{t} \, =\, \sum\limits_{k=-\infty}^{\infty} \varphi_{t}^{k} \cdot \varphi_{t}^{j-k}\quad \text{otherwise},
\end{array}
\right.
\end{equation}
with $\, \varphi^{0}_{T} \, =\, c\,\big(pp_1+(1-p)q_1\big)\,$, $\,\varphi^{-1}_{T} \, =\, - c(1-p)\,q_1\,$, $\,\varphi^{+1}_{T} \, =\, - cp\,p_1\,$, $\,\varphi^{j}_{T} \equiv 0 \,$, $\,j \neq -1, 0, 1 \,$. 
\begin{remark} \label{fullydierected_sum0}
According to equation \eqref{reduced ODE}, the sum $\, \sum\limits_{j=-\infty}^{\infty} \varphi^{j}_{t}\,$ satisfies 
\begin{equation}
\frac{\, {\mathrm d} \,}{\, {\mathrm d} t \,} \sum_{j=-\infty}^{\infty} {\varphi}^{j}_{t} \, =\, \big( \sum_{j=-\infty}^{\infty} {\varphi}^{j}_{t} \big)^{2}\, 
, \quad \sum_{j=-\infty}^{\infty} {\varphi}^{j}_{T} \, =\,  0 \, . 
\end{equation}
This ordinary differential equation  has a unique solution 
\begin{equation}
\sum_{j=-\infty}^{\infty} {\varphi}^{j}_{t} \, \equiv \, 0 \, ; \quad 0 \le t \le T \, . 
\end{equation} 
\end{remark}

\bigskip
The generating function $\,S_{t}(z) \, :=\, \sum\limits_{k=-\infty}^{\infty} z^{k} \varphi_{t}^{k} \,$, $\,z \in \mathbb C \setminus \{0\} \,$, if it is well defined (and the superscript $\,j\,$ of $\,z^{j}\,$ is the power), satisfies the one-dimensional Riccati equation 
\begin{equation}\label{Riccati}
\begin{split}
\dot{S}_{t}(z) \, &=\,  \sum_{j=-\infty}^{\infty} z^{j} \dot{\varphi}^{k}_{t} \, =\,  \sum_{j, k=-\infty}^{\infty} \varphi_{t}^{k} \varphi_{t}^{j-k} z^{j} + \frac{\,1\,}{\,z\,} \cdot \varepsilon (1-p) q_1+ z\cdot \varepsilon p\,p_1-\varepsilon\, \big(pp_1+(1-p)q_1\big)\\
 \,& =\,  \sum_{k=-\infty}^{\infty} \sum_{\ell=-\infty}^{\infty} \varphi_{t}^{k} \varphi_{t}^{\ell} z^{k+\ell} - \Big( 1 - \frac{\,1\,}{\,z\,} \Big) \varepsilon (1-p)q_1- ( 1- z) \varepsilon pp_1\\
 \,  & =\, [ S_{t}(z)]^{2} - \Big[  \Big( 1 - \frac{\,1\,}{\,z\,} \Big) \, \varepsilon (1-p)q_1 + ( 1- z) \, \varepsilon pp_1 \Big] \, \\
  \,  & =\, [ S_{t}(z)]^{2} -\varepsilon\, T(z)\, ; \quad z \in \mathbb C \, , \, \, 0 \le t \le T \, 
\end{split}
\end{equation}
with $\, S_{T}(z) \, =\, c\,T(z)$, where $T(z)=\, ( 1 - \frac{1}{z}) \, (1-p)q_1 + ( 1- z) \, pp_1\,$. 

\noindent $\,\bullet\,$ For $z^{\pm}=\, \frac{\big(pp_1+(1-p)q_1\big)\,\pm\, \sqrt{\big(pp_1+(1-p)q_1\big)^2-4pp_1\,(1-p)q_1}}{2pp_1}=1\,\,\text{or}\,\, \dfrac{(1-p)q_1}{pp_1}$, $T(z^{\pm})=0$. Then we get the ODE:
$\Dot{S}_t(z^{\pm})=(S_t(z^{\pm}))^2\, ,\,S_T(z^{\pm})=0.$
The solution is $S_t(z^{\pm})=0$. Then we can conclude:
\begin{equation*}
\, \sum_{k=-\infty}^{\infty} (z^{\pm})^{k} \varphi_{t}^{k} \,=\,0.
\end{equation*}
\noindent $\,\bullet\,$ For $\, z \in \mathbb C \setminus \{ 0 \}\,$ and $\,z\neq \, 1\,\,\text{or}\,\, \dfrac{(1-p)q_1}{pp_1}$, the solution $\, S_{t}(z) \,$ is given by 
\begin{equation} \label{eq: Stz-2}
\begin{split}
S_{t}(z) \, &=\, \sum_{k=-\infty}^{\infty} z^{k} \varphi^{k}_{t} \, =\, \mathfrak b(z) \cdot  \frac{\,  ( \mathfrak b(z) + \mathfrak q (z)) \cdot e^{ \mathfrak b(z) (T-t) } -  ( \mathfrak b(z) - \mathfrak q (z)) \cdot   e^{- \mathfrak b(z) (T-t) }\,}{\, ( \mathfrak b(z) + \mathfrak q (z)) \cdot e^{\mathfrak b(z) (T-t) } + ( \mathfrak b(z) - \mathfrak q (z)) \cdot e^{ - \mathfrak b(z) (T- t) } \,} \, \quad\xrightarrow[T\to \infty]{} \mathfrak b (z); \\
\mathfrak b (z)\, &:=\,\sqrt{\varepsilon\, T(z)}\,=\,  \Big[ \Big( 1 - \frac{\,1\,}{\,z\,} \Big) \, \varepsilon (1-p) q_1+ (1- z) \, \varepsilon p p_1\Big]^{1/2} \, , \\
\mathfrak q (z) \,  &:=\,c\, T(z)\,=\, \Big( 1 - \frac{\,1\,}{\,z\,} \Big) \, c(1-p) q_1+ ( 1- z) \, cpp_1 \, .
\end{split} 
\end{equation}
\subsubsection{Stationary Solution for Two-sided Directed Chain Game}
In this section, we want to see how the values $\,p\,,\,p_1\, ,\,q_1\,$ affect the game. For our analysis let us consider the limits $\, {\bm \phi}^{j}_{t} \, :=\,  \lim_{T \to \infty} \varphi^{j}_{t} \,$of $\,\varphi_{t}^{j}\,$ for $\,t \ge 0 \,$, $\, j \in \mathbb Z\,$, as $\, T \to \infty\,$ and take them as a stationary solution of \eqref{reduced ODE}. As $\, T \to \infty\,$, we have obtained from \eqref{eq: Stz-2} 
\[
\,\lim_{T\to \infty} \sum_{j=-\infty}^{\infty} z^{j} \varphi^{j}_{t} \, =\,  \lim_{T\to \infty} S_{t}(z) \, =\,   \mathfrak b(z)  \, =\,  \sum_{j=-\infty}^{\infty} z^{j} {\bm \phi}^{j}_{t}\,; \quad \, t \ge 0 \, , \, \, z \in \{ z : \mathfrak b (z) \in \mathbb R , \mathfrak b ( z) > 0 \} \,, 
\] 
where $\, \mathfrak b(\cdot) \,$ does not depend on $\, t \,$. Hence, the limit $\, {\bm \phi}^{j}_{t} \,$ does not depend on $\, t\,$, and we write it as constant $\, {\bm \phi}^{j}\,$ for every $\, j \in \mathbb Z\,$. Also, substituting this observation into \eqref{eq: ODEs} with $\, \psi_{T}^{i} = 0 \,$, we observe $\, \psi_{t}^{i} \equiv  0 \,$, and hence, we obtain a dynamics for the stationary equilibrium: 
\begin{equation} \label{eq: NE} 
{\mathrm d} X_{t}^{i} \, =\,  - \sum_{k=-\infty}^{\infty} {\bm \phi}^{i-k} X_{t}^{k} {\mathrm d} t + \sigma {\mathrm d} W_{t}^{i} \, ; \quad X_{0}^{i} \, =\,  \xi^{i} \, ; \quad i \in \mathbb Z \,, \quad t \ge 0  \, . 
\end{equation}
We shall identify the values $\,{\bm \phi}^{j}\,$, $\,j \in \mathbb Z\,$ and behaviors of $\, X_{\cdot}^{i}\,$, $\, i \in \mathbb Z \,$. 

The function $\, \mathfrak b(z) \,$ can be rewritten as 
\begin{equation} \label{eq: b 1D}
\begin{split} 
\mathfrak b(z) \, = &\, \sqrt{\varepsilon} \, \Big[ \Big( 1 - \frac{\,1\,}{\,z\,} \Big) \, (1-p)q_1  + (1- z) \,pp_1 \Big]^{1/2}\\ 
\, = &\, \sqrt{\varepsilon} \cdot \sqrt{pp_1+(1-p)q_1}\,\Big[ 1 - \Big( z \, \dfrac{p \,p_1}{pp_1+(1-p)q_1}+ \frac{\,1 \,}{\,z\,}\,\dfrac{(1-p)q_1}{pp_1+(1-p)q_1}\Big) \Big]^{1/2} \\
\,=&\, \sqrt{\varepsilon\,\big(pp_1+(1-p)q_1\big)} \,\Big[ 1 - \Big( z \, w+ \frac{\,v \,}{\,z\,}\Big) \Big]^{1/2}
\end{split}
\end{equation}
for $\, z \in \mathbb C \setminus \{ 0 \} \,\cap\{{ \frac{1\,\pm\, \sqrt{1-4wv}}{2w}}\} $ and define $w\,=\, \frac{pp_1}{pp_1+(1-p)q_1}\,,\, v\,=\,\frac{(1-p)q_1}{pp_1+(1-p)q_1}$, where $0<p<1,\,0\leq p_1\leq 1$ and $0\leq q_1\leq 1$. First, by inequalities, we have $wv=\dfrac{pp_1(1-p)q_1}{\big(pp_1+(1-p)q_1\big)^2}\,\in\,[0,\frac{1}{4}].$

\noindent $\bullet$ When $wv=0$, i.e. $p_1=0$ or $q_1=0$, we get back to Section \ref{section random chain}, one direction random chain game. For example, when $q_1=0$, each player is interacted with her/his neighbor with a probability of $pp_1$.

\noindent $\bullet$ In the case when $wv\,=\, pp_1\,(1-p)q_1 \,\in\,(0, \frac{1}{4}]\,$, i.e. $0<p_1\leq 1$ and $0<q_1\leq 1$, we expand formally
\begin{equation}
\begin{split}
\mathfrak b(z) \, = & \,  \sqrt{\varepsilon\,\big(pp_1+(1-p)q_1\big)} \,\sum_{i=0}^{\infty} {1/2 \choose i } ( -1)^{i} \Big ( z \, w +  \frac{\,  v \, }{z} \Big)^{i} \\
\, =& \,\sqrt{\varepsilon\,\big(pp_1+(1-p)q_1\big)} \, \sum_{i=0}^{\infty} {1/2 \choose i } (-1)^{i} \sum_{k=0}^{i} {i \choose k} w^{k} v^{i-k} z^{2 k - i } \, =\,  \sum_{j=-\infty}^{\infty} z^{j} {\bm \phi}^{j}\, , 
\end{split}
\end{equation}
and hence, comparing the coefficients of $\,z^{j}\,$ and letting $B=\,pp_1+(1-p)q_1\,$, we obtain  
\begin{equation} \label{eq: bphi1}
\begin{split}
\,{\bm \phi}^{0} \, = & \,\sqrt{\varepsilon\,B} \,\sum_{\ell=0}^{\infty} {1/2 \choose 2 \ell } { 2 \ell  \choose \ell  } ( -1)^{2\ell} w ^{\ell} v^{\ell} \, = \,  _{2} F_{1} ( - 1/4, 1/4, 1 , 4 wv ) \,, \\
\end{split}
\end{equation}
\begin{equation} \label{eq: bphi2}
\begin{split}
{\bm \phi}^{j} \, = &\, \sqrt{\varepsilon\,B} \,\sum_{\ell=0}^{\infty} {1/2 \choose 2 \ell + j } { 2 \ell + j \choose \ell + j } ( -1)^{2\ell + j} w ^{\ell +j} v^{\ell}  \\
\, =&\,\sqrt{\varepsilon\,B} \, (-1)^{j} w^{j} {1/2 \choose j }\,  \, _{2} F_{1} \Big( - \frac{\,1\,}{\,4\,} + \frac{\,j\,}{\,2\,}, \frac{\,1\,}{\,4\,} + \frac{\,j\,}{\,2\,}, 1 + j , 4 wv \Big)\, ,  \\
{\bm \phi}^{-j} \, = &\, \sqrt{\varepsilon\,B} \,\sum_{\ell=0}^{\infty} {1/2 \choose 2 \ell + j } { 2 \ell + j \choose  \ell } ( -1)^{2\ell + j} w ^{\ell +j} v^{\ell} \\
\, = & \, \sqrt{\varepsilon\,B} \,( -1)^{j} w^{j} v^{ \frac{\,1\,}{\,1\,} - \,j\,} {1/2 \choose j } \cosh \Big( \Big( j - \frac{\,1\,}{\,2\,} \Big) \tanh^{-1} ( \sqrt{wv } ) \Big) \,  
\end{split}
\end{equation}
for $\, j \ge 1\,$, where $\,\tanh^{-1}(\cdot) \,$ is the inverse hyperbolic tangent function and $\, _{2}F_{1}(\cdot)\,$ is the hypergeometric function defined by 
\[
_{2} F_{1}(a_{1}, a_{2}; b_{1} ; z) \, :=\, \sum_{n=0}^{\infty} \frac{\,(a_{1})_{n} \cdot  (a_{2})_{n}\,}{\,(b_{1})_{n} \, \cdot \, n ! \,} \cdot z^{n} \, ; \quad z \in \mathbb C 
\]
with the rising factorial $\, (a)_{0} \, =\,  1\,$, $\, (a)_{n} \, =\,  a(a+1) \cdots (a+n-1) \,$ for $\, a \in \mathbb C\,$, $\, n \ge 1\,$.

\subsubsection{Special Case: Catalan Markov Chain of the Deterministic Two-sided Chain Game}

When the chain is deterministic, i.e., $\,p_1\,=\,q_1\,=\,1$, the stationary solution is give in \eqref{eq: bphi1} - \eqref{eq: bphi2} by taking $\,w=\frac{pp_1}{pp_1+(1-p)q_1}=p\,$, $\,v=\frac{(1-p)q_1}{pp_1+(1-p)q_1}=1-p\,$ and $B=pp_1+(1-p)q_1=1$. 

\begin{remark}
When the chain is deterministic and the interaction is symmetric, i.e. $\,p_1\,=\,q_1\,=\,1$ and $p=\frac{1}{2}$, solutions \eqref{eq: bphi1} - \eqref{eq: bphi2} suggest to take 
$\, w=\frac{pp_1}{pp_1+(1-p)q_1} = 1/2 \,$ and $\, v=\frac{(1-p)q_1}{pp_1+(1-p)q_1} = 1/2\,$, and we obtain simpler forms: 
\begin{equation} \label{eq: bphi3}
{\bm \phi}^{0} \, =\,  \frac{\,2 \sqrt{ 2\,  \varepsilon } \,}{\,\pi \,} \, , \quad {\bm \phi}^{j} \, =\, (-1)^{j} \cdot \sqrt{\frac{\,2 \varepsilon\,}{\,\pi \,}} { 1/2 \choose j } \frac{\,\Gamma ( 1 + j ) \,}{\, \Gamma ( (3 + 2 j ) / 2) \,} \, , 
\end{equation}
\[
{\bm \phi}^{-j} \, =\, (-1)^{j} \cdot \frac{\,\sqrt{\varepsilon} ( 3 j + \sqrt{3} ) \,}{\,2 \sqrt{2}\, 3^{j}\,} \, {1/2 \choose j } \, ; \quad j \ge 1 \, . 
\]
\end{remark}

\bigskip

Coming back to general $\, p \in (0, 1) \,$, we have by numerical evaluation,
\[
{\bm \phi}^{0} \, >\,  0  \, , \quad {\bm \phi}^{j} < 0 \, , \quad j \in \mathbb Z \, , 
\]
and hence, \eqref{eq: NE} can be seen as a linear evolution equation. Without loss of generality, we assume $\varepsilon=1$ and $\sigma=1$. Since we have the relation : $\sum\limits_{k=-\infty}^\infty {\bm \phi}^{k}=0$ in Remark \ref{fullydierected_sum0}, we can consider a continuous-time Markov chain $M(\cdot)$ in the state space $\, \mathbb Z\,$ with generator matrix $\mathbf{Q}\,=\,-\,\left( \begin{array}{cccccccc} 
\ddots & \ddots & \ddots &\ddots  & \ddots&\ddots  & \ddots\\
\ddots & \cdots &  {\bm \phi}^{-1} & {\bm \phi}^0 & {\bm \phi}^1  & \ddots & \ddots & \ddots\\
\ddots & {\bm \phi}^{-k}  &  \cdots & {\bm \phi}^{-1}  & {\bm \phi}^0 & {\bm \phi}^1  & \ddots & \ddots\\
\ddots & {\bm \phi}^{-(k+1)}  & \cdots & \cdots  & {\bm \phi}^{-1}  & {\bm \phi}^0 & {\bm \phi}^1 &  \ddots \\
& \ddots & \ddots &\ddots  & \ddots&\ddots  & \ddots & \ddots\\
\end{array} 
\right)$. 
The infinite particle system \eqref{eq: NE} can be represented as a stochastic evolution equation:
\begin{equation}  \label{eqn:evolution}
{\mathrm d}\mathbf{X}_t=\mathbf{Q\,X}_t{\mathrm d}t+{\mathrm d}\mathbf{W}_t,
\end{equation}
where $\mathbf{X}_{\cdot}=(X_{\cdot}^{k},k\in \mathbb{Z})$ with $\,\mathbf{X_0}={\bm \xi} \, :=\, (\xi^{i}, i \in \mathbb Z ) \,$, $\,{\bm W}_{\cdot} \, :=\, ( W_{\cdot}^{i}, i \in \mathbb Z ) \,$. The solution is:
\begin{equation}
\mathbf{X}_t={\bm \xi}e^{t\mathbf{Q}}+ \int_0^t e^{(t-s)\mathbf{Q}}{\mathrm d}\mathbf{W}_s ;\quad t\geq 0,
\end{equation}
where $\, e^{u {\mathbf{Q}}}\,$, $\,u \ge 0 \,$ forms the semigroup induced by the continuous-time Markov chain with the generator $\,\mathbf{Q} \,$ and the transition probability matrix function $\, p_{i,j}(t)\,$, $\, i, j \in \mathbb Z\,$, $\, t \ge 0 \,$ in the state space $\, \mathbb Z\,$. Without loss of generality, let us assume $\mathbf{X}_0=\mathbf{0}$. With these transition probability matrix function, we may write the solution of \eqref{eq: NE} as 
\begin{equation}
X_{t}^{i} \, =\int^{t}_{0} \sum_{j=-\infty}^{\infty} { p}_{i,j}(t-s) {\mathrm d}W_{s}^{j} \, ; \quad i \in \mathbb Z \, , \, \, t \ge 0 \, .  
\end{equation} 

The variance of $\, X_{t}^{i}\,$ is given by 
\begin{equation}
\text{Var} ( X_{t}^{i}) \, = \int^{t}_{0} \sum_{j=-\infty}^{\infty} [ {p}_{i,j}(t-s)]^{2}  {\mathrm d} s < \infty \, . 
\end{equation}

\begin{prop}\label{sol_markovctwo}
The Gaussian process $\, X_{t}^{0} \,$,
 $\, t \ge 0 \,$ in \eqref{eqn:evolution}, corresponding to the (Catalan) Markov chain, is given by
\begin{equation} \label{solution_mc}
\begin{split}
{X}_{t}^{0}  \, &=\, \sum\limits_{j=-\infty}^{\infty}\int^{t}_{0}   (\exp ( Q (t-s)))_{0,j} {\mathrm d} {W}_s^j \,\\
&=\sum\limits_{\ell=0}^\infty\sum\limits_{m=-\ell}^\ell \,\int^{t}_{0}\frac{(t-s)^{4\ell} F^{(2\ell)}(-(t-s)^2)}{(2\ell)!}\,\binom{2\ell}{\ell+m}\,p^{\ell+m}(1-p)^{\ell-m}{\mathrm d} {W}_s^{2m}\\
&+\sum\limits_{\ell=0}^\infty \sum\limits_{m=-(\ell+1)}^\ell\,\int^{t}_{0} \frac{(t-s)^{4\ell+2} F^{(2\ell+1)}(-(t-s)^2)}{(2\ell+1)!}\, \binom{2\ell+1}{\ell+1+m}\,p^{\ell+1+m}(1-p)^{\ell-m}  {\mathrm d} {W}_s^{2m+1}\,,
\end{split}
\end{equation}
where $\, {W}_{\cdot}^{j}  \,$, $ j \in \mathbb Z¸ \,$ are independent standard Brownian motions and $\,F^{(k)}(x) \, =\,  \rho_{k}(x) e^{ - \sqrt{-x}}\,$, with
\[
\rho_k(x)=\frac{1}{2^k} \sum\limits_{j=k}^{2k-1}\, \frac{(j-1)!}{(2j-2k)!!(2k-j-1)!} \, (-x)^{\,-\frac{j}{2}},\quad \textit{for} \quad k\geq 1.
\]
and $\rho_{0}(x) \, =\,  1$ for $x \le 0 $.    
\end{prop}
\begin{proof}
Given in \textbf{Appendix} \ref{appendix_CatalanMarkovChainTwo}.
\end{proof}

\section{Random Directed Tree Game}\label{section-tree-model}
Motivated by the discussion about the deterministic directed infinite tree game in Feng, Fouque \& Ichiba \cite{fengFouqueIchiba2020linearquadratic}, we now look at a random tree structure.
The connection and similarity between random and non-random cases is illustrated in Corollary \ref{random-deterministic cases}.

\subsection{Setup and Assumptions}\label{random tree model}
We describe a stochastic game on a directed tree where the interaction between every two players in the neighboring generation is random. All players have a fixed number of potential players in the next generation to interact with, denoted by a finite positive integer $M$. That is, for $n,k\geq 1$, player $(n,k)$ is the $k$-th individual of the $n$-th generation and she can interact with the players in the $(n+1)$-th generation labelled as $\{(n+1,M(k-1)+j): 1\leq j\leq M\}$. We introduce the {\it i.i.d.} binary random variables $N^{n,k,M(k-1)+j}$, valued in $\{0,1\}$,  which represent the random interaction between player $(n,k)$ and player $(n+1,M(k-1)+j)$ for $1\leq j\leq M$, present with probability  $p$, where $0<p<1$. When $N^{n,k,M(k-1)+j}$ is zero, we assume player $(n,k)$ has no interaction with player $(n+1,M(k-1)+j)$ for $1\leq j\leq M$.
We assume the dynamics of the states of the players are given by the stochastic differential equations of the form:
\begin{equation}\label{dynamic-random-tree}
    {\mathrm d}X_t^{n,k}=\a_t^{n,k}  {\mathrm d}t+\s  {\mathrm d}W_t^{n,k}, \quad 0\leq t\leq T,
\end{equation}
where $(W_t^{n,k})_{0\leq t\leq T},n,k\geq 1$ are one-dimensional independent standard Brownian motions. We assume that the diffusion is one-dimensional and the diffusion coefficients are constant and identical denoted by $\sigma>0$. The drift coefficients $\alpha^{n,k}$'s are adapted to the filtration of the Brownian motions and satisfy $\mathbbm{E}[\int_0^T |\alpha_t^{n,k}|^2 dt]<\infty$. 
The system starts at time $t = 0$ from $i.i.d.$ square-integrable random variables $X_0^{n,k} = \xi_{n,k}$ independent of the Brownian motions and, without loss of generality,  we assume ${\mathbbm E}(\xi_{n,k}) = 0$ for every pair of $(n,k)$. 

In this model, each player $({n,k})$ chooses its own strategy $\alpha^{n,k}$ in order to minimize its objective function of the form:

\begin{align}\label{firstObjective}
J^{n,k}(\boldsymbol{\alpha})=
&\mathbbm{E}_{N,X} \bigg\{ \displaystyle \int_0^T \bigg(\frac{1}{2}(\a^{n,k}_t)^2\\
&\quad\quad\quad+\frac{\e}{2}\Big(\dfrac{1}{\sum\limits_{j=1}^M N^{n,k, M(k-1)+j}}\,\,\sum_{j=1}^{M}N^{n,k, M(k-1)+j}\,X_t^{n,k, M(k-1)+j}-X_t^{n,k}\Big)^2\cdot \mathbbm{1}_{\sum\limits_{j=1}^M N^{n,k, M(k-1)+j}\neq 0}\bigg) {\mathrm d}t \nonumber\\
&\quad\quad\quad+\frac{c}{2}\Big(\dfrac{1}{\sum\limits_{j=1}^M N^{n,k, M(k-1)+j}}\,\,\sum_{j=1}^{M}N^{n,k, M(k-1)+j}\,X_T^{n,k, M(k-1)+j}-X_T^{n,k}\Big)^2\cdot \mathbbm{1}_{\sum\limits_{j=1}^M N^{n,k, M(k-1)+j}\neq 0}\bigg\}. \nonumber
\end{align}
for some constants $\e>0$, $c\geq 0$ and $\boldsymbol{\alpha}=(\alpha^{n,k}: n\geq 1,1\leq k\leq M^{n-1})$ with $\alpha^{n,k}\in \mathbb{R}$. When the player has no connection with any player in the next generation, her insentive is to choose $\alpha^{n,k}=0$.

Conditioning on $\sum\limits_{j=1}^M N^{n,k, M(k-1)+j}=d_{n,k}$ where $0\leq d_{n,k}\leq M$, and denoting $p_{d_{n,k}}=P(\sum\limits_{j=1}^M N^{n,k, M(k-1)+j}=d_{n,k})=\displaystyle \binom{M}{d_{n,k}}p^{d_{n,k}}\,(1-p)^{M-d_{n,k}}$, we get
\begin{align}
J^{n,k}(\boldsymbol{\alpha})
=\mathbbm{E}_{X} \bigg\{ \displaystyle \int_0^T \bigg(\frac{1}{2}(\a^{n,k}_t)^2
&+\frac{\e}{2}\sum_{d_{n,k}=1}^M p_{d_{n,k}}\cdot
 \dfrac{1}{\binom{M}{d_{n,k}}}\,\sum\limits_{I\in S_{d_{n,k}}} \Big(\frac{1}{d_{n,k}}\sum_{j\in I} X_t^{n+1,j}-X_t^{n,k}\Big)^2\bigg) {\mathrm d}t\\
&+\frac{c}{2}\sum_{d_{n,k}=1}^M p_{d_{n,k}}\cdot\dfrac{1}{\binom{M}{d_{n,k}}}\, \sum\limits_{I\in S_{d_{n,k}}} \Big(\frac{1}{d_{n,k}}\sum_{j\in I}X_T^{n+1,j}-X_T^{n,k}\Big)^2\bigg\}, \nonumber
\end{align}
where $S_{d_{n,k}}=\{(i_1,\cdots,i_{d_{n,k}}):M(k-1)+1\leq i_1<\cdots<i_{d_{n,k}}\leq Mk\}$ denotes the set of all possible combinations of $d_{n,k}$ elements between $M(k-1)+1$ and $Mk$ with an increasing order.

\subsection{Open-Loop Nash Equilibrium} \label{Section 1}
We search for an open-loop Nash equilibrium of the directed random tree system among strategies $\{\a^{n,k};n\geq 1, k\geq 1\}$.
The Hamiltonian for player $(n,k)$ is of the form:
\begin{align*}
&H^{n,k}(x^{m,l},y^{n,k;m,l},\a^{m,l};m\geq 1, 1\leq l\leq M^{m-1})
=\displaystyle\sum\limits_{m=1}^{N_n}\,\sum\limits_{l=1}^{M^{m-1}} \a^{m,l}y^{n,k;m,l}
+
 \frac{1}{2}(\a^{n,k})^2
 \\
&\quad\quad\quad\quad+\displaystyle \frac{\e}{2}\sum_{d_{n,k}=1}^M p_{d_{n,k}}\cdot
 \dfrac{1}{\binom{M}{d_{n,k}}}\,\sum\limits_{I\in S_{d_{n,k}}}\,\Big(\frac{1}{d_{n,k}}\sum_{j\in I} x^{n+1,j}-x^{n,k}\Big)^2,
\end{align*}
assuming it is defined on $Y^{n,k}_t$'s where {\it only finitely many}  $Y^{n,k;m,l}_t$'s are non-zero for every given $(n,k)$. Here, $N_n$ represents a depth of this finite dependence, a finite number depending on $n$ with $N_n>n$ for $n \ge 1$. This assumption is checked in Remark \ref{finitecheck_random tree} below. Thus, the Hamiltonian $H^{n,k}$ for player $(n,k)$ is well defined for $n \ge 1$.

The adjoint processes $Y_t^{n,k}=(Y_t^{n,k;m,l};m\geq 1, 1\leq l\leq M^{m-1})$ and $Z_t^{n,k}=(Z_t^{n,k;m,l;p,q};m,p\geq 1, 1\leq l\leq M^{m-1},1\leq q\leq M^{p-1})$ for $n\geq 1,1\leq k\leq M^{n-1}$ are defined as the solutions of the backward stochastic differential equations (BSDEs):
\begin{align}
 {\mathrm d}Y_t^{n,k;m,l}&=-\partial_{x^{m,l}}H^{n,k} (X_t,Y_t^{n,k},\a_t){\mathrm d}t+ \displaystyle \sum\limits_{p=1}^\infty \sum\limits_{q=1}^{M^{p-1}} Z_t^{n,k;m,l;p,q}{\mathrm d}W_t^{p,q} \label{eq: last BSDEs}\\
     &=-\e\displaystyle \sum_{d_{n,k}=1}^M p_{d_{n,k}}\cdot
 \dfrac{1}{\binom{M}{d_{n,k}}}\,\sum\limits_{I\in S_{d_{n,k}}}\,\Big(\frac{1}{d_{n,k}}\sum_{j\in I}\,X_t^{n+1,j}-X_t^{n,k}\Big) \Big(\frac{1}{d_{n,k}}\displaystyle\sum_{j\in I}\,\delta_{(m,l),(n+1,j)}-\delta_{(m,l),(n,k)}\Big){\mathrm d}t\nonumber \\
    &\quad+ \displaystyle \sum\limits_{p=1}^\infty \sum\limits_{q=1}^{M^{p-1}} Z_t^{n,k;m,l;p,q}{\mathrm d}W_t^{p,q},\nonumber
    \end{align}
    with terminal condition: 
    \begin{align*}
    Y_T^{n,k;m,l}
    & = c\cdot\displaystyle \sum_{d_{n,k}=1}^M p_{d_{n,k}}\cdot
 \dfrac{1}{\binom{M}{d_{n,k}}}\,\sum\limits_{I\in S_{d_{n,k}}}\,\Big(\frac{1}{d_{n,k}}\sum_{j\in I}\,X_T^{n+1,j}-X_T^{n,k} \Big)\Big(\frac{1}{d_{n,k}}\displaystyle\sum_{j\in I}\,\delta_{(m,l),(n+1,j)}-\delta_{(m,l),(n,k)}\Big).
\end{align*}

\begin{remark}\label{finitecheck_random tree}
For every $(m,l)\neq (n,k)$ or $(n+1,i)$ for $M(k-1)+1\leq i \leq Mk$, $ {\mathrm d}Y_t^{n,k;m,l}=\sum\limits_{p=1}^\infty \sum\limits_{q=1}^{M^{p-1}} Z_t^{n,k;m,l;p,q}{\mathrm d}W_t^{p,q}$ and $Y_T^{n,k;m,l}=0$ implies $Z_t^{n,k;m,l;p,q}=0$ for all $(p,q)$. Thus there are finitely many non-zero $Y^{n,k;m,l}$'s for every $(n,k)$ and the Hamiltonian can be rewritten as

\begin{equation}
\begin{split} 
H^{n,k}&(x^{m,l},y^{n,k;n,k},y^{n,k;n+1,i},\a^{m,l};m\geq 1, 1\leq l\leq M^{m-1}, M(k-1)+1\leq i\leq Mk)\\
=&\a^{n,k}y^{n,k;n,k}+\sum_{i=M(k-1)+1}^{Mk} \a^{n+1,i}y^{n,k;n+1,i}
+\frac{1}{2}(\a^{n,k})^2
\\
& \hspace{3cm} {} +\displaystyle\frac{\e}{2}\sum_{d_{n,k}=1}^M p_{d_{n,k}}\cdot
 \dfrac{1}{\binom{M}{d_{n,k}}}\,\sum\limits_{I\in S_{d_{n,k}}}\,\Big(\frac{1}{d_{n,k}}\sum_{j\in I}\, x^{n+1,j}-x^{n,k}\Big)^2.
\end{split} 
\end{equation}
\end{remark}

\begin{remark}
For every $(n,k)$, we will solve \eqref{eq: last BSDEs} when $(m,l)=(n,k)$ in the following discussion. Other non-zero $Y^{n,k;m,l}$'s in \eqref{eq: last BSDEs} are solvable with the similar method.

\end{remark}

When $(m,l)=(n,k)$, \eqref{eq: last BSDEs}  becomes:
\begin{equation} \label{BSDEsum_tree}
 \left\{
    \begin{array}{ll}
        {\mathrm d}Y_t^{n,k;n,k}&=\e\displaystyle \sum_{d_{n,k}=1}^M p_{d_{n,k}}\cdot
 \dfrac{1}{\binom{M}{d_{n,k}}}\,\sum\limits_{I\in S_{d_{n,k}}}\,\Big(\frac{1}{d_{n,k}}\sum_{j\in I}\,X_t^{n+1,j}-X_t^{n,k}){\mathrm d}t +\sum_{p=1}^\infty \sum_{q=1}^{M^{p-1}} Z_t^{n,k;n,k;p,q}{\mathrm d}W_t^{p,q},\\
    Y_T^{n,k;n,k}&=-c\displaystyle \sum_{d_{n,k}=1}^M p_{d_{n,k}}\cdot
 \dfrac{1}{\binom{M}{d_{n,k}}}\,\sum\limits_{I\in S_{d_{n,k}}}\,\Big(\frac{1}{d_{n,k}}\sum_{j\in I}\, X_T^{n+1,j}-X_T^{n,k}\Big).
    \end{array}
    \right.
\end{equation}

To simplify the equation system, we use the result: for all $d_{n,k}\in [M(k-1)+1, Mk]$
\begin{align*}
\quad &p_{d_{n,k}}\cdot
 \dfrac{1}{\binom{M}{d_{n,k}}}\,\sum\limits_{I\in S_{d_{n,k}}}\,\frac{1}{d_{n,k}}\sum_{j\in I}\,x^{n+1,j}\,=
p_{d_{n,k}}\,\frac{1}{d_{n,k}}\,  \dfrac{1}{\binom{M}{d_{n,k}}}\,
\sum\limits_{I=(i_1,\cdots,i_{d_{n,k}})\in S_{d_{n,k}} }\,(x^{n+1,i_1}+\cdots+x^{n+1,i_{d_{n,k}}})\\
&\quad\quad\quad =p_{d_{n,k}}\,\frac{1}{d_{n,k}}\,  \dfrac{1}{\binom{M}{d_{n,k}}}\,
\cdot \binom{M-1}{d_{n,k}-1}
 \,(x^{n+1,M(k-1)+1}+\cdots+x^{n+1,Mk}) \\
 &\quad\quad\quad = p_{d_{n,k}}\frac{1}{M}\,
 \,\sum\limits_{j=M(k-1)+1}^{Mk} x^{n+1,j},
\end{align*}
which gives:
\begin{align*}
\quad &\sum_{d_{n,k}=1}^M p_{d_{n,k}}\cdot
 \dfrac{1}{\binom{M}{d_{n,k}}}\,\sum\limits_{I\in S_{d_{n,k}}}\,\frac{1}{d_{n,k}}\sum_{j\in I}\,x^{n+1,j}
 =(1-p_0)\frac{1}{M}\,
 \,\sum\limits_{j=M(k-1)+1}^{Mk} x^{n+1,j}.
\end{align*}

Then we can rewrite system \eqref{BSDEsum_tree} as:
\begin{equation}\label{BSDE-tree}
 \left\{
    \begin{array}{ll}
        {\mathrm d}Y_t^{n,k;n,k}&=\e(1-p_0) \displaystyle \Big(\frac{1}{M} \sum_{j=M(k-1)+1}^{Mk}X_t^{n+1,j}-X_t^{n,k}\Big){\mathrm d}t +\sum_{p=1}^\infty \sum_{q=1}^{M^{p-1}} Z_t^{n,k;n,k;p,q}{\mathrm d}W_t^{p,q},\\
    Y_T^{n,k;n,k}&=-c(1-p_0)\displaystyle \Big(\frac{1}{M} \sum_{j=M(k-1)+1}^{Mk}X_T^{n+1,j}-X_T^{n,k} \Big).
    \end{array}
    \right.
\end{equation}

By minimizing the Hamiltonian with respect to $\a^{n,k}$, we can get an open-loop Nash equilibrium:  $\hat{\a}^{n,k}=-y^{n,k;n,k}$ for all $(n,k)$. Considering the BSDE system, we make the ansatz of the form:
\begin{equation}\label{ansatz-tree}
 Y_t^{n,k;n,k}=\sum\limits_{i=n}^\infty\sum\limits_{j=M^{i-n}(k-1)+1}^{M^{i-n}k} \phi_t^{n,k;\,i,j}X_t^{i,j},
\end{equation}
for some deterministic scalar function $\phi_t$ depending on $(n,k)$. 
According to \eqref{BSDE-tree}, the functions satisfy the terminal conditions: 
\begin{align*}
\noindent \,\bullet\,&\phi_T^{n,k;\,n,k}=c(1-p_0); \\
\noindent \,\bullet\,&\phi_T^{n,k;n+1,j}=-c\displaystyle (1-p_0)\frac{1}{M} ,\,\text{ for }\,M(k-1)+1\leq j\leq Mk;\\
\noindent \,\bullet\,&\phi_T^{n,k;\,n+\ell,j}=0,\quad \text{for } \ell\geq 2,M^{\ell}(k-1)+1 \leq j\leq M^\ell k.
\end{align*}

Using the ansatz, the optimal strategy $\hat{\alpha}^{n,k}$  and the forward equation for $X_{\cdot}^{n,k}$ in (\ref{dynamic-random-tree}) become:
\begin{equation}\label{particle_system_tree}
 \left\{
  \begin{array}{ll}
    \hat{\a}_t^{n,k}&=-Y_t^{n,k;n,k}=- \displaystyle \sum\limits_{i=n}^\infty\sum\limits_{j=M^{i-n}(k-1)+1}^{M^{i-n}k} \phi_t^{n,k;\,i,j}X_t^{i,j},\\
     {\mathrm d}X_t^{n,k}&=- \displaystyle\sum\limits_{i=n}^\infty\sum\limits_{j=M^{i-n}(k-1)+1}^{M^{i-n}k} \phi_t^{n,k;\,i,j}X_t^{i,j} {\mathrm d}t+\s  {\mathrm d}W_t^{n,k},
  \end{array}
  \right.
\end{equation}
which gives: for $0 \le t \le T$
\[
{\mathrm d}X_t^{i,j}= - \sum\limits_{r=i}^\infty\sum\limits_{s=M^{r-i}(j-1)+1}^{M^{r-i}j} \phi_t^{i,j;\,r,s}\,X_t^{r,s} {\mathrm d}t+\s {\mathrm d}W_t^{i,j}.
\]

Define a set $S^{u,v}=\{i:M^u(v-1)+1\leq i\leq M^uv,\, i \in \mathbb{N}\}$. Differentiating the ansatz (\ref{ansatz-tree}) and substituting \eqref{particle_system_tree}, we obtain:
\begin{equation}\label{ito-tree}
 \begin{split}
    {\mathrm d}Y_t^{n,k;n,k}&=\sum\limits_{i=n}^\infty\sum\limits_{j\in S^{i-n,k}}(\dot{\phi}_t^{n,k;\,i,j}X_t^{i,j}{\mathrm d}t+\phi_t^{n,k;\,i,j}{\mathrm d} X_t^{i,j})\\
    &=\sum\limits_{i=n}^\infty\sum\limits_{j\in S^{i-n,k}}\dot{\phi}_t^{n,k;\,i,j}X_t^{i,j}{\mathrm d}t+\sum\limits_{i=n}^\infty\sum\limits_{j\in S^{i-n,k}} \phi_t^{n,k;\,i,j} \big[ - \sum\limits_{r=i}^\infty\sum\limits_{s\in S^{r-i,j}}\phi_t^{i,j:\,r,s}X_t^{r,s} {\mathrm d}t+\s {\mathrm d}W_t^{i,j}\big]\\
    &=\sum\limits_{i=n}^\infty\sum\limits_{j\in S^{i-n,k}}\dot{\phi}_t^{n,k;\,i,j}X_t^{i,j}{\mathrm d}t-\sum\limits_{i=n}^\infty \sum\limits_{r=i}^\infty\sum\limits_{j\in S^{i-n,k}}\,\sum\limits_{s\in S^{r-i,j}} \phi_t^{n,k;\,i,j}  \phi_t^{i,j;\,r,s}X_t^{r,s} {\mathrm d}t+\s \sum\limits_{i=n}^\infty\sum\limits_{j\in S^{i-n,k}}\phi_t^{n,k;\,i,j} {\mathrm d}W_t^{i,j}\\
     &\stackrel{\text{def}}{=}\text{I}  -\text{II}+\text{III}.
 \end{split}
\end{equation}
For the first and third terms, we have 
\begin{align*}
\text{I}&=\sum\limits_{i=n}^\infty\sum\limits_{j\in S^{i-n,k}}\dot{\phi}_t^{n,k;\,i,j}X_t^{i,j}{\mathrm d}t= \sum\limits_{r=n}^\infty\sum\limits_{s=M^{r-n}(k-1)+1}^{M^{r-n}k} \dot{\phi}_t^{n,k;\,r,s}X_t^{r,s}{\mathrm d}t;\\
\text{III}&=\s \sum\limits_{i=n}^\infty\sum\limits_{j\in S^{i-n,k}}\phi_t^{n,k;\,i,j} {\mathrm d}W_t^{i,j}=\s\sum\limits_{r=n}^\infty\sum\limits_{s=M^{r-n}(k-1)+1}^{M^{r-n}k} \phi_t^{n,k;\,r,s} {\mathrm d}W_t^{r,s}. 
\end{align*}

Then, for the second term, we have 
\begin{align*}
\text{II}&=\sum\limits_{i=n}^\infty \sum\limits_{r=i}^\infty\sum\limits_{j\in S^{i-n,k}}\,\sum\limits_{s\in S^{r-i,j}} \phi_t^{n,k;\,i,j}  \phi_t^{i,j;\,r,s}X_t^{r,s} {\mathrm d}t=\sum\limits_{r=n}^\infty\sum\limits_{s=M^{r-n}(k-1)+1}^{M^{r-n}k}\,\bigg(\sum\limits_{i=n}^r  \phi_t^{n,k;\,i,\left\lceil \frac{s}{M^{r-i}}\right\rceil}  \phi_t^{i,\left\lceil \frac{s}{M^{r-i}}\right\rceil;\,r,s}\bigg)X_t^{r,s} {\mathrm d}t,\\
&\text{where} \left\lceil x\right\rceil\,\text{ denotes here the smallest integer greater than or equal to}\, x. 
\end{align*}

Thus  equation \eqref{ito-tree} can be written as:
\begin{align}
{\mathrm d}Y_t^{n,k;n,k}& = \text{I}  -\text{II}+\text{III}\label{ito-tree-second}\\
&=\sum\limits_{r=n}^\infty\sum\limits_{s=M^{r-n}(k-1)+1}^{M^{r-n}k}\bigg(\dot{\phi}_t^{n,k;\,r,s}- \sum\limits_{i=n}^r  \phi_t^{n,k;\,i,\left\lceil \frac{s}{M^{r-i}}\right\rceil}  \phi_t^{i,\left\lceil \frac{s}{M^{r-i}}\right\rceil;\,r,s}\bigg)X_t^{r,s} {\mathrm d}t+\s\sum\limits_{r=n}^\infty\sum\limits_{s=M^{r-n}(k-1)+1}^{M^{r-n}k} \phi_t^{n,k;\,r,s} {\mathrm d}W_t^{r,s}.\nonumber
\end{align}

Now comparing the two It\^o's decompositions (\ref{BSDE-tree}) and (\ref{ito-tree-second}), we obtain first the processes 
$Z_t^{n,k;n,k;p,q}$ from the martingale terms :
\begin{equation*}
   Z_t^{n,k;n,k;p,q}=\s\phi_t^{n,k;\,p,q} \text{ for } p\geq n \text{ and } M^{p-n}(k-1)+1\leq q\leq M^{p-n}k\, ;\quad Z_t^{n,k;n,k;p,q}=0,  \text{ otherwise}.
\end{equation*}

Then we obtain from the drift terms:
\begin{flalign}
	\noindent \,\bullet\,     \quad 
    &\Dot{\phi}_t^{n,k;\,n,k}
      = \phi_t^{n,k;\,n,k}\phi_t^{n,k;\,n,k}-\e(1-p_0), \quad\quad\phi_T^{n,k;\,n,k}=c\displaystyle (1-p_0);\label{eqn:tmp1}\\
      &{\Longrightarrow} \phi_t^{n,k;\,n,k}\equiv \phi_t^{i,j;\,i,j} \, \text{ for any pairs } \, (n,k),\,(i,j);\nonumber &&
  \end{flalign}
\begin{flalign}
    \noindent \,\bullet\,& \text{for }\, M(k-1)+1\leq \ell\leq Mk,\nonumber\\
    &\quad \Dot{\phi}_t^{n,k;\,n+1,\ell}
    = \phi_t^{n,k;\,n,k}\phi_t^{n,k;\,n+1,\ell}+\phi_t^{n,k;\,n+1,\ell}\phi_t^{n+1,\ell;\,n+1,\ell}+\e\displaystyle (1-p_0)\frac{1}{M} \nonumber &&\\
    &\quad\quad\quad\quad\quad \stackrel{\eqref{eqn:tmp1}}{=} 2 \phi_t^{n,k;\,n,k}\phi_t^{n,k;\,n+1,\ell}+\e\displaystyle (1-p_0)\frac{1}{M},\quad\quad  \phi_T^{n,k;\,n+1,\ell}=-\displaystyle c(1-p_0)\frac{1}{M} ;\label{eqn:tmp2}
    \end{flalign}
\begin{flalign}
\noindent \,\bullet\, & \text{for }\, m\geq n+2,\,M^{m-n}(k-1)+1\leq \ell\leq M^{m-n}k,\nonumber\\
&\quad  \Dot{\phi}_t^{n,k;\,m,\ell}=\sum\limits_{i=n}^{m} \phi_t^{n,k;\,i,\left\lceil \frac{\ell}{M^{m-i}}\right\rceil}  \phi_t^{i,\left\lceil \frac{\ell}{M^{m-i}}\right\rceil;\,m,\ell} , \quad\quad \phi_T^{n,k;\,m,\ell}=0. \label{eqn:tmp3}&&
\\\nonumber
\end{flalign}

\subsection{Discussion about the Solution}
\begin{remark}
Since by definition, $p_0=(1-p)^M=\mathbb{E}\big[\mathbbm{1}_{\{ \sum\limits_{j=1}^M N^{n,k, M(k-1)+j}\neq 0 \} }\big]$ for any $(n,k)$, the above equation system \eqref{eqn:tmp1}-\eqref{eqn:tmp3} depends on $p$ and $M$.
\end{remark}

\begin{theorem} \label{theorem: 1}
The solution of the system \eqref{eqn:tmp1}-\eqref{eqn:tmp3} are independent of $n,k$ and depend on the "depth" $i$, i.e.,  
$$\phi_t^{n,k;\,n+i,\ell}=\phi_t^{m,j;\,m+i,\tilde{\ell}}, \quad t \ge 0 $$ for every suitable pairs $(n,k)$, $(m,j)$ and every suitable $i, \ell, \widetilde{\ell} \in \mathbb N_{0}$. Thus the system is closed for $\{\phi_t^{n,k;\,m,\ell},\,m\geq n,\, M^{m-n}(k-1)+1\leq \ell\leq M^{m-n}k\}$ and the solutions exist. 
\end{theorem}

\begin{proof}
First, \eqref{eqn:tmp1} is a simple Riccati equation for $\phi^{n,k;n,k}$ and it is  independent of $(n,k)$. Thus, its solution $\phi_t^{n,k;\,n,k}$ exists uniquely for every $(n,k)$ with $ \phi_t^{n,k;\,n,k}\equiv \phi_t^{m,j;\,m,j}$, $t \ge 0 $ for any suitable pairs $(n,k),\,(m,j)$. This is depth $0$.

Next, substituting $ \phi_t^{n,k;\,n,k}\equiv \phi_t^{m,j;\,m,j}$ into the first line of \eqref{eqn:tmp2}, we see for every $M(k-1)+1\leq \ell\leq Mk$,  \eqref{eqn:tmp2} is a first-order linear differential equation for $\phi^{n,k ; n+1, \ell}$ and it depends only on $p_0$, $\varepsilon$ and $M$ but not on $(n,k)$. Thus, we claim that the solution $\{\phi_t^{n,k;\,n+1,\ell+M(k-1)}, 1\leq \ell\leq M\}$ of \eqref{eqn:tmp2} exists uniquely. They are identical among the depth $1$, i.e.,  
\begin{equation}\label{eqn:phi^(n,n+1) relation}
\quad \phi_t^{n,k;\,n+1,\ell+M(k-1)}\equiv \phi_t^{m,j;\,m+1,\tilde{\ell}+M(j-1)} \quad \text{for }\quad 1\leq \ell,\tilde{\ell}\leq M.
\end{equation}
For $m=n+2$ and $1\leq \tilde{\ell}\leq M^{2}$ in \eqref{eqn:tmp3}, we have the derivative of the function ${\phi}_t^{n,k;\,n+2,\tilde{\ell}+M^{2}(k-1)}$ of depth $2$: 
\begin{align*}
\Dot{\phi}_t^{n,k;\,n+2,\tilde{\ell}+M^{2}(k-1)}&=\sum\limits_{i=n}^{n+2} \phi_t^{n,k;\,i,\left\lceil \frac{\tilde{\ell}}{M^{n+2-i}}\right\rceil +M^{i-n}(k-1)}  \phi_t^{i,\left\lceil \frac{\tilde{\ell}}{M^{n+2-i}}\right\rceil +M^{i-n}(k-1);\,n+2,\tilde{\ell}+M^{2}(k-1)}\\
&=\phi_t^{n,k;\,n,k}\,\phi_t^{n,k;\,n+2,\tilde{\ell}+M^{2}(k-1)}+ \phi_t^{n,k;\,n+1,\left\lceil \frac{\tilde{\ell}}{M}\right\rceil+M(k-1)}  \phi_t^{n+1,\left\lceil \frac{\tilde{\ell}}{M}\right\rceil+M(k-1);\,n+2,\tilde{\ell}+M^{2}(k-1)}\\
&\quad{}+ \phi_t^{n,k;\,n+2, \tilde{\ell}+M^{2}(k-1)}  \phi_t^{n+2, \tilde{\ell}+M^{2}(k-1);\,n+2,\tilde{\ell}+M^{2}(k-1)}\\
&=2\phi_t^{n,k;\,n,k}\,\phi_t^{n,k;\,n+2,\tilde{\ell}+M^{2}(k-1)}+ \phi_t^{n,k;\,n+1,\left\lceil \frac{\tilde{\ell}}{M}\right\rceil+M(k-1)}  \phi_t^{n+1,\left\lceil \frac{\tilde{\ell}}{M}\right\rceil +M(k-1);\,n+2,\tilde{\ell}+M^{2}(k-1)}, 
\end{align*}
where  the last term has the function of the depth $1$: 
\begin{align*}
\phi_t^{n+1,\left\lceil \frac{\tilde{\ell}}{M}\right\rceil +M(k-1);\,n+2,\tilde{\ell}+M^{2}(k-1)}
&=\phi_t^{n+1,\left\lceil \frac{\tilde{\ell}}{M}\right\rceil +M(k-1);\,n+2,\tilde{\ell}^{\prime}+M\big[\left\lceil \frac{\tilde{\ell}}{M}\right\rceil +M(k-1)-1\big]}\\
&=\phi_t^{n,k;\,n+1,\tilde{\ell}^{\prime}+M(k-1)}\quad\text{according to equation }\eqref{eqn:phi^(n,n+1) relation},
\end{align*}
where $\tilde{\ell}^{\prime}=\tilde{\ell}+M-M\left\lceil \frac{\tilde{\ell}}{M}\right\rceil $ satisfies $1 \leq \tilde{\ell}^{\prime}\leq M$. Thus, the differential equation for $\phi_t^{n,k;\,n+2,\tilde{\ell}+M^{2}(k-1)}$ is reduced to a first order linear differential equation depending on the functions $\phi_t^{n,k;\,n,k}$ and $\phi_t^{n,k;\,n+1,\tilde{\ell}+M(k-1)}$  of depths $\,1, 2\,$. We claim the solutions $\{\phi_t^{n,k;\,n+2,\tilde{\ell}+M^{2}(k-1)},\, 1\leq \tilde{\ell}\leq M^2\}$ of depth $2$ exist and are identical. 

When $m>n+2$ in  \eqref{eqn:tmp3}, we proceed the discussion recursively by using a similar method, that is, we can reduce the differential equation for the function $\phi_t^{n,k;\,m,\tilde{\ell}+M^{m-n}(k-1)}$ of depth $m-n$ to a first order linear differential equation depending on functions $\{\phi_t^{n,k;\,i,\tilde{\ell}+M^{i-n}(k-1)},n\leq i<m, 1\leq \tilde{\ell}\leq M^{i-n}\}$ of shallower depths less than $m-n$. Then we verify  that the solutions $\{\phi_t^{n,k;\,m,\tilde{\ell}+M^{m-n}(k-1)},\, 1\leq \tilde{\ell}\leq M^{m-n}\}$ of \eqref{eqn:tmp3} exist uniquely and identical among the depth $m-n$ for every given $m > n + 2$. 

From above, we can conclude the solution $\phi^{n, \cdot; m, \cdot}$ of the system exist, depending only on the "depth" $m-n$. Therefore, the system \eqref{eqn:tmp1}-\eqref{eqn:tmp3} can be reduced to a closed system for $\{\phi_t^{n,k;\,m,\ell},\,m\geq n,\, M^{m-n}(k-1)+1\leq \ell\leq M^{m-n}k\}$.
\end{proof}

\begin{remark}\label{random-deterministic cases}
If we assume that each individual $(n,k)$ is in interaction with its all potential players in the $(n+1)$-th generation, i.e., $p=1,\,p_0=0$, the model becomes the same as the system of the deterministic tree model in in Feng, Fouque \& Ichiba \cite{fengFouqueIchiba2020linearquadratic}. 

\begin{proof}
When $N_{n,k}\equiv M (\,>0)$, the functions only depend on the depth. Thus equations \eqref{eqn:tmp1}-\eqref{eqn:tmp3} above become
\begin{align*}
&\dot{\Psi}_t^{m}: = \dot{\phi}_t^{n,k;\,n+m,\ell}=\sum\limits_{i=n}^{n+m} \Psi_t^{i-n}\cdot \Psi_t^{n+m-i}-\delta_{m,0}\e+\delta_{m,1}\e\frac{1}{M}=\sum\limits_{i=0}^{m}\Psi_t^{i}\cdot \Psi_t^{m-i}-\delta_{m,0}\e+\delta_{m,1}\e\frac{1}{M},\\
& \Psi_T^{m}=\delta_{m,0}c-\delta_{m,1}c\frac{1}{M}.
\end{align*}
It is the same as the Riccati system of the deterministic tree model in Feng, Fouque \& Ichiba \cite{fengFouqueIchiba2020linearquadratic} with $M=d$.\\
\end{proof}
\end{remark}

As a consequence of Theorem \ref{theorem: 1}, the infinite-player stochastic game on the random tree model has an open-loop Nash equilibrium: 

\begin{prop} An open-loop Nash equilibrium for the infinite-player stochastic game on the random tree with cost functionals \eqref{firstObjective} is determined by  \eqref{particle_system_tree}, where $\{ \phi^{n,k;i,j}, i \geq n,\, M^{i-n}(k-1)+1\leq j\leq M^{i-n}k \} $ are the unique solution to the infinite system \eqref{eqn:tmp1}-\eqref{eqn:tmp3} of Riccati equations.  
\end{prop}

\section{Conclusion}\label{conclusion}

We studied a linear-quadratic stochastic differential game on a random directed chain network by assuming the interaction between every two neighbors exists with a probability $p$. We constructed an open-loop Nash equilibria in the case of infinite chain and computed the stationary solution explicitly, named Catalan functions. The equilibrium is characterized by interactions with all the players in one direction of the chain weighted by Catalan functions and the probability of interaction $p$. The asymptotic variance of a player's state converges to a finite limit depending on $p$ in the infinite time limit, which is different from the behavior of the nearest neighbor dynamics discussed in Detering, Fouque \& Ichiba \cite{Nils-JP-Ichiba2018DirectedChain}. In the particular case with the probability of interaction equal to $1$, we obtain the deterministic directed chain structure studied in Feng, Fouque \& Ichiba \cite{fengFouqueIchiba2020linearquadratic}. The random directed game model is extended to games on a random two-sided directed chain structure and a random tree structure.

\appendix
\section{Appendix}\label{Appendix}

\subsection{Stationary Solution of  the Riccati System \eqref{random chain riccati-1} }\label{moment_generating_method}
By taking $T\to \infty$ and assuming $\epsilon=1$, the constant solution of the moment generating function \eqref{Sol S_t} satisfying $\Dot{S}_t(z)=0$ is $S(z)=\sqrt{p(1-z)}$. We can then find constant solutions for $\phi$ functions by taking Taylor expansion and comparing it with $S(z)=\sum_{k=0}^\infty z^k\ \phi^{(k)}$, because 
\begin{equation*}
S(z)=\sqrt{p(1-z)}=\sqrt{p}\,\sqrt{1-z}=\sqrt{p}\,\sum\limits_{k=0}^\infty  \binom{\frac{1}{2}}{k} \Big(-z\Big)^k 
=\sqrt{p}-\frac{\sqrt{p}}{2} z-\sqrt{p}\sum\limits_{k=2}^\infty \frac{(2k-3)!! }{2^k k!} z^k. 
\end{equation*}

\subsection{Proof of Proposition \ref{sol_markovc_random}} \label{appendix_CatalanMarkovChainRandom}
We have the results: $q_0=-1$, $q_1=\dfrac{1}{2}$, $\sum\limits_{j=0}^k q_jq_{k-j}=0$ for $k\geq 2$. Then, it is easily seen that $\,(\sqrt{p}\, \mathbf{Q})^{2} \, =p\, ( I - B)\,$ with $\, B\,$ having $\,1\,$'s on the upper second diagonal and $\,0\,$'s elsewhere, i.e., 
\[
(\sqrt{p}\,\mathbf{Q})^{2} \, =\, \left( \begin{array}{ccccc} 
p & -p &  0 & \cdots & \\
0 & p & -p &  \ddots &  \\
& \ddots & \ddots &\ddots  & \\
\end{array} 
\right) \, = \, - p\,J_{\infty} (-1)  \, , \quad J_{\infty} (\lambda) \, :=\,  \left( \begin{array}{ccccc} 
\lambda & 1 &  0 & \cdots & \\
0 & \lambda & 1 &  \ddots &  \\
& \ddots & \ddots &\ddots  & \\
\end{array} 
\right) \, . 
\]
Here, $\, J_{\infty}(\lambda) \,$ is the infinite Jordan block matrix with diagonal components $\,\lambda\,$. 

The matrix exponential of $\, \sqrt{p}\,\mathbf{Q} t\,$, $\, t \ge 0 \,$, is written formally as
\[
\exp ( \sqrt{p}\,\mathbf{Q} t ) \, =\,  F(-p\,\mathbf{Q}^{2} t^{2}) \, =\,  F ( J_{\infty} (-1) \cdot p\,t^{2} ) \, , \, \, t \ge 0 \, ,  \quad F(x) \, :=\, \exp ( - \sqrt{ - x } )\, ,  \, \, x \in \mathbb C \, . 
\]
Since a smooth function of a Jordan block matrix can be expressed as 
\[
F(J_{\infty} (\lambda) ) \, =\, F(\lambda I+B)=\,\sum\limits_{k=0}^\infty \frac{F^{(k)}(\lambda)}{k!}\, B^k =\, \left( \begin{array}{cccccc} 
F(\lambda) & F^{(1)}(\lambda) & \frac{\,F^{(2)}(\lambda) \,}{\,2!\,} & \cdots & \frac{\,F^{(k)}(\lambda) \,}{\,k!\,} & \cdots \\
& \ddots & \ddots & \ddots & & \ddots \\
 & & \ddots & \ddots & \ddots & \\ 
\end{array} \right ) \, ,
\]
we get
\[
\exp (\sqrt{p}\,\mathbf{Q} t )=\, F(J(-\infty)\cdot p\,t^2)=\, F((-I+B)\cdot p\,t^2)=\,\sum\limits_{k=0}^\infty \frac{F^{(k)}(-pt^2)}{k!}\, (B\,pt^2)^k=\,\sum\limits_{k=0}^\infty \frac{p^k\,t^{2k}F^{(k)}(-pt^2)}{k!}\, B^k  .
\]
The $\,(j,k)\,$-element of $\, \exp ( \sqrt{p}\,\mathbf{Q} t ) \,$ is formally given by 
\[
(\exp (\sqrt{p}\,\mathbf{Q} t ))_{j,k} \, =\,  \frac{p^{k-j}\,t^{2(k-j)} \cdot F^{(k-j)}(-pt^{2}) \,}{\, (k-j)!\,} \,  , \quad j \le k \, , \, \, \text{ where } \, \, F^{(k)} (x) \, :=\, \frac{\,{\mathrm d}^{k}  F \,}{\,{\mathrm d} x^{k}\,}(x)  \, ; \quad x > 0 \,, \, \, k \in \mathbb N \,  ,   
\]
and $\, (\exp ( \sqrt{p}\,\mathbf{Q} t))_{j,k} \, =\,  0 \,$, $\, j > k \,$ for $\, t \ge 0 \,$. Here the $\,k\,$-th derivative $\, F^{(k)}(x) \,$ of $\, F(\cdot) \,$ can be written as $\,F^{(k)}(x) \, =\,  \rho_{k}(x) e^{ - \sqrt{-x}}\,$, where $\, \rho_{k}(x) \,$ satisfies the recursive equation 
\[
\rho_{k+1}(x) \, =\, \rho^{\prime}_{k}(x) +\frac{\,\rho_{k}(x)\,}{\,2 \sqrt{ - x} \,} \, ; \quad k \ge 0 \, , \,  
\]
with $\, \rho_{0}(x) \, =\,  1 \, $, $\, x \in \mathbb C \,$. 
By mathematical induction, we may verify
\begin{eqnarray}
\rho_k(x) 
                &=&\frac{1}{2^k} \sum\limits_{j=k}^{2k-1}\, \frac{(j-1)!}{(2j-2k)!!(2k-j-1)!} \, (-x)^{\,-\frac{j}{2}}, \quad 
                k\geq 1. 
\end{eqnarray}

Therefore, substituting them into \eqref{solution_mc_random}, we obtain the formula of Gaussian process.

Next, it follows from (\ref{solution_mc_random}) that for $\, t \ge 0 \,$, the variance of the Gaussian process $X_{\cdot}^{i}$, $i \ge 1$ is given by  
\begin{equation} \label{eq: VarX_random}
\begin{split}
\text{Var} ( X_{t}^{i}) = \text{Var} ( X_{t}^{1}) \, &=\, \text{Var} \Big(  \sum_{j=1}^{\infty} \int^{t}_{0}  \frac{p^{j-1}\,(t-s)^{2(j-1)} \,}{\,(j-1)!\,} F^{(j-1)}(-p(t-s)^{2}) {\mathrm d} W_{s}^{j} \Big)  
\\
\, &=\,  \sum_{j=0}^{\infty} \int^{t}_{0} \frac{p^{2j}\,(t-s)^{4j}\,}{\,(j!)^{2} \,} \lvert  \rho_{j} (- p(t-s)^{2}) \rvert^{2} e^{-2\sqrt{p}(t-s)}{\mathrm d} s.
\end{split}
\end{equation}
Since it can be shown that
\begin{equation} \label{eq: rho-k-nu2}
\rho_j(-\nu^2)=\frac{\,1\,}{\,2^{j} \nu^{j}\,}\cdot \sqrt{\frac{\,2 \nu \,}{\,\pi \,} } \cdot  e^{\nu} \cdot K_{j-(1/2)} (\nu ) \, ; \quad j \ge 1 \, ,  
\end{equation}
where $K_n(x)$ is the modified Bessel function of the second kind defined by 
\[
K_{n}(x)\, =\,  \int^{\infty}_{0} e^{-x \cosh t} \cosh (nt ) {\mathrm d} t \, ; \quad n > -1 ,  x > 0 .  
\]

Then substituting \eqref{eq: rho-k-nu2} into \eqref{eq: VarX_random} and using the change of variables, we obtain 
\begin{equation*} 
\begin{split}
\text{\rm Var} ( X_{t}^{1}) 
&=\dfrac{1}{\sqrt{p}}\, \sum_{k=1}^{\infty}  \int^{\sqrt{p}\,t}_{0} \frac{\,2\,}{\,\pi\,}  \frac{\,\nu^{2k+1}\,}{\,(k!)^{2}\, 4^{k}\,} \big( K _{k-(1/2)}(\nu) \big)^{2} {\mathrm d} \nu + \frac{\,1 - e^{-2\sqrt{p}t}\,}{\,2\sqrt{p}\,}; \quad t \ge 0 . 
\end{split}
\end{equation*}

Using the following identities from the special functions
\begin{equation*}
\begin{split} 
&\int_0^\infty t^{\alpha-1}(K_\nu(t))^2 dt= \frac{\sqrt{\pi}}{4\Gamma ((\alpha+1)/2 )} \Gamma\Big(\frac{\alpha}{2}\Big)\Gamma\Big(\frac{\alpha}{2}-\nu\Big)\Gamma \Big(\frac{\alpha}{2}+\nu\Big),\\
\\
&\frac{\sqrt{2}}{4}x \sqrt{x^2-\sqrt{x^4-16}}=
\sum\limits_{k=0}^\infty {4k \choose 2k}\frac{1}{2k+1}\frac{1}{x^{4k}},\quad \text{for } x\geq 2, \\
\end{split} 
\end{equation*}
we obtain the limit of variance of $X_{t}^{1}$, as $\,t \to \infty\,$, i.e.,  
\begin{equation*}
\begin{split}
\lim_{t\to \infty} \text{Var} ( X_{t}^{1}) =& \frac{\,1\,}{\,2\sqrt{p}\,} + \dfrac{1}{\sqrt{p}}\,\sum_{k=1}^{\infty} \int^{\infty}_{0} \frac{\,2\,  s^{2k+1}\,}{\,\pi (k!)^{2} 4^{k}\,} \cdot [ K_{k-(1/2)}(s)]^{2} {\mathrm d} s 
=\, \frac{\,1\,}{\,2\sqrt{p}\,} + \dfrac{1}{\sqrt{p}}\,\sum_{k=1}^{\infty} \frac{\,2\,}{\,\pi \, (k!)^{2} 4^{k}\,} \int^{\infty}_{0} s^{2k+1} [ K_{k-(1/2)}(s)]^{2}{\mathrm d} s \,\\
\, \\
&=\, \frac{\,1\,}{\,2\sqrt{p}\,} + \dfrac{1}{\sqrt{p}}\,\sum_{k=1}^{\infty} \frac{\,2 \,}{\,\pi (k!)^{2} 4^{k}\,} \cdot \frac{\, \pi  \, \Gamma ( k + 1) \,  \Gamma ( 2k + (1/2)) \,}{\, 8\,  \Gamma ( k + (3/2)) \,}
=\, \frac{1}{2\,\sqrt{p}} \sum_{k=0}^{\infty}   {4k \choose 2k} \frac{1}{2k+1} \frac{1}{2^{4k}} \,\\
\, &=\,  \frac{\,1\,}{\,2\sqrt{p}\,}\, \cdot \, \frac{\sqrt{2}}{4}2\sqrt{2^2-0}=\frac{1}{\sqrt{2p}}.  
\end{split}
\end{equation*}

\subsection{Proof of Proposition \ref{sol_markovctwo}} \label{appendix_CatalanMarkovChainTwo}
We assume $\varepsilon=1$, $p_k=- {\bm \phi}^{k}\,=\,  \lim_{T \to \infty} \varphi^{k}_{t}$. According to the equations (\ref{reduced ODE}), for the fully directed two-sided chain and , we have: 
$\sum\limits_{k=-\infty}^\infty p_k\,p_{-k}=1,\sum\limits_{k=-\infty}^\infty p_k\,p_{1-k}=-p,\sum\limits_{k=-\infty}^\infty p_k\,p_{-1-k}=-(1-p), \sum\limits_{k=-\infty}^\infty p_k\,p_{j-k}=0$ for other $j$. Then it is easily seen that $\, \mathbf{Q}^{2} \, =\,  I - (pB^*+(1-p)B_*)\,$ with $\, B^*\,$ having $\,1\,$'s on the upper second diagonal and $\,0\,$'s elsewhere, and $\, B_*\,$ having $\,1\,$'s on the lower second diagonal and $\,0\,$'s elsewhere i.e., 
\[
\mathbf{Q}^2\,=\left( \begin{array}{ccccccc} 
 \ddots & \ddots &\ddots  & \ddots&\ddots  & \ddots & \ddots\\
\ddots & -(1-p) & 1 & -p &  0 &\ddots&\ddots\\
\ddots & 0 & -(1-p) &  1 & -p &  0 &\ddots \\
\ddots &  \ddots  & 0 &-(1-p) &  1 & -p & \ddots\\
 \ddots & \ddots &\ddots  & \ddots&\ddots  & \ddots & \ddots\\
\end{array} 
\right)\,,
\]
\[
B^*=\left( \begin{array}{ccccccc} 
 \ddots & \ddots &\ddots  & \ddots&\ddots  & \ddots & \ddots\\
\ddots & 0 & 0 & 1 &  0 &\ddots&\ddots\\
\ddots & 0 & 0 &  0 & 1 &  0 &\ddots \\
\ddots &  \ddots  & 0 &0 &  0 & 1 & \ddots\\
 \ddots & \ddots &\ddots  & \ddots&\ddots  & \ddots & \ddots\\
\end{array} 
\right) \, ,\quad
B_*=\left( \begin{array}{ccccccc} 
 \ddots & \ddots &\ddots  & \ddots&\ddots  & \ddots & \ddots\\
\ddots & 1 & 0 & 0 &  0 &\ddots&\ddots\\
\ddots & 0 & 1 &  0 & 0 &  0 &\ddots \\
\ddots &  \ddots  & 0 &1 &  0 & 0 & \ddots\\
 \ddots & \ddots &\ddots  & \ddots&\ddots  & \ddots & \ddots\\
\end{array} 
\right) \, . 
\]

If we look at the power of $pB^*+(1-p)B_*$:
\[
\begin{split}
& pB^*+(1-p)B_*=\left( \begin{array}{ccccccc} 
 \ddots & \ddots &\ddots  & \ddots&\ddots  & \ddots & \ddots\\
\ddots & 0 & 1-p & 0 & p &  0 &\ddots \\
 \ddots & \ddots &\ddots  & \ddots&\ddots  & \ddots & \ddots\\
\end{array} 
\right) \, ;\\
&(pB^*+(1-p)B_*)^2=\left( \begin{array}{ccccccc} 
 \ddots & \ddots &\ddots  & \ddots&\ddots  & \ddots & \ddots\\
\ddots & (1-p)^2 & 0 & 2p(1-p) & 0 &  p^2 &\ddots \\
 \ddots & \ddots &\ddots  & \ddots&\ddots  & \ddots & \ddots\\
\end{array} 
\right) \, ;\\
& (pB^*+(1-p)B_*)^3=\left( \begin{array}{ccccccccc} 
 \ddots & \ddots &\ddots  & \ddots&\ddots  & \ddots & \ddots& \ddots\\
\ddots & (1-p)^3 & 0 & 3p(1-p)^2 & 0 &  3p^2(1-p) & 0 & p^3 &\ddots \\
 \ddots & \ddots &\ddots  & \ddots&\ddots  & \ddots & \ddots& \ddots\\
\end{array} 
\right);\\
&(pB^*+(1-p)B_*)^4=\left( \begin{array}{ccccccccccc} 
 \ddots & \ddots &\ddots  & \ddots&\ddots  & \ddots & \ddots& \ddots & \ddots& \ddots & \ddots\\
\ddots & (1-p)^4 & 0 & 4p(1-p)^3 & 0 &  6p^2(1-p)^2 & 0 & 4p^3(1-p) &0&p^4 &\ddots \\
 \ddots & \ddots &\ddots  & \ddots&\ddots  & \ddots & \ddots& \ddots & \ddots& \ddots & \ddots\\
\end{array} 
\right) \, ;\\
&\quad \cdots\cdots\cdots\cdots
\end{split}
\]

We find the diagonal increases following the binomial expansion and we have formulas to generalize the result:
\[
\left\{
\begin{split}
k \text{ even}: & (pB^*+(1-p)B_*)^k=\left( \begin{array}{ccccccc} 
 \ddots & \ddots &\ddots  & \ddots&\ddots  & \ddots & \ddots\\
\ddots & \binom{k}{\frac{k}{2}-1}\, p^{\frac{k}{2}-1}(1-p)^{\frac{k}{2}+1} & 0 &  \binom{k}{\frac{k}{2}}\,p^{\frac{k}{2}}(1-p)^{\frac{k}{2}}& 0 &  \binom{k}{\frac{k}{2}+1}\,p^{\frac{k}{2}+1}(1-p)^{\frac{k}{2}-1} &\ddots \\
 \ddots & \ddots &\ddots  & \ddots&\ddots  & \ddots & \ddots\\
\end{array} 
\right)\\
k \text{ odd}: & (pB^*+(1-p)B_*)^k=\left( \begin{array}{ccccccc} 
 \ddots & \ddots &\ddots  & \ddots&\ddots  & \ddots & \ddots\\
\ddots & 0 & \binom{k}{\frac{k+1}{2}-1}\,p^{\frac{k+1}{2}-1}(1-p)^{\frac{k-1}{2}+1} &  0& \binom{k}{\frac{k+1}{2}}\,p^{\frac{k+1}{2}}(1-p)^{\frac{k-1}{2}} &  0 &\ddots \\
 \ddots & \ddots &\ddots  & \ddots&\ddots  & \ddots & \ddots\\
\end{array} 
\right)
\end{split}
\right.
\]
i.e.
\[\begin{split}
& k \text{ even}: ((pB^*+(1-p)B_*)^k)_{i,j}=\left\{\begin{split}
&\quad \binom{k}{\frac{k}{2}+m}\,\,p^{\frac{k}{2}+m}(1-p)^{\frac{k}{2}-m} ,\quad j=i+2m,\text{ and } -\frac{k}{2}\leq m\leq \frac{k}{2}\quad m\in\mathbb{Z};\\
&\quad \quad 0,\quad \quad\quad \quad\quad \quad\quad \quad \quad \quad\quad \text{otherwise}. 
\end{split}\right.\\
\\
&k \text{ odd}: ((pB^*+(1-p)B_*)^k)_{i,j}=\left\{\begin{split}
&\quad \binom{k}{\frac{k+1}{2}+m}\,\,p^{\frac{k+1}{2}+m}(1-p)^{\frac{k-1}{2}-m} ,\quad j=i+2m+1,\text{ and } -\frac{k+1}{2}\leq m\leq \frac{k-1}{2};\\
&\quad \quad 0,\quad \quad\quad \quad\quad \quad\quad \quad \quad \quad\quad \text{otherwise}. 
\end{split}\right.
\end{split}
\]

The matrix exponential of $\, \mathbf{Q} t\,$, $\, t \ge 0 \,$ is written formally by 
\[
\exp ( \mathbf{Q} t ) \, =\,  F(-\mathbf{Q}^{2} t^{2}) \, , t \ge 0 \, ,  \quad F(x) \, :=\, \exp ( - \sqrt{ - x } )\, ,  \, \, x \in \mathbb C \, . 
\]
Since a smooth function can be expressed as 
\[
F(\lambda I+B)=\,\sum\limits_{k=0}^\infty \frac{F^{(k)}(\lambda)}{k!}\, B^k =\sum\limits_{k\,\, odd} \frac{F^{(k)}(\lambda)}{k!}\, B^k + \sum\limits_{k\,\, even} \frac{F^{(k)}(\lambda)}{k!}\, B^k.
\]

So 
\[
\exp (\mathbf{Q} t )=\, F((-I+pB^*+(1-p)B_*)t^2)=\,\sum\limits_{k=0}^\infty \frac{F^{(k)}(-t^2)}{k!}\, \big((pB^*+(1-p)B_*)t^2\big)^k=\,\sum\limits_{k=0}^\infty \frac{t^{2k}F^{(k)}(-t^2)}{k!}\, (pB^*+(1-p)B_*)^k .
\]
The $\,(i,j)\,$-element of $\, \exp ( \mathbf{Q} t ) \,$, is formally given by 
\[
(\exp (\mathbf{Q} t ))_{i,j} \, =\left\{\begin{split}
&\sum\limits_{k\,\, even} \,\frac{t^{2k} F^{(k)}(-t^2)}{k!}\,\binom{k}{\frac{k}{2}+m} \,p^{\frac{k}{2}+m} (1-p)^{\frac{k}{2}-m}\cdot \mathbbm{1}_{-\frac{k}{2}\leq m\leq \frac{k}{2}} , &\quad j = i+2m, \,m\in\mathbb{Z} ,\\
&\sum\limits_{k\,\, odd}  \,\frac{t^{2k} F^{(k)}(-t^2)}{k!}\,\binom{k}{\frac{k+1}{2}+m}\,p^{\frac{k+1}{2}+m}(1-p)^{\frac{k-1}{2}-m}\cdot \mathbbm{1}_{-\frac{k+1}{2}\leq m\leq \frac{k-1}{2}}, &\quad  j=i+2m+1, \,m\in\mathbb{Z} .
\end{split}  
\right.
\]
\[=
\left\{
\begin{split}
&\sum\limits_{\ell=0}^\infty  \,\frac{t^{4\ell} F^{(2\ell)}(-t^2)}{(2\ell)!}\,\binom{2\ell}{\ell+m}\,p^{\ell+m}(1-p)^{\ell-m}\cdot \mathbbm{1}_{-\ell\leq m\leq \ell}, &\quad j = i+2m,\,m\in\mathbb{Z},\\
&\sum\limits_{\ell=0}^\infty  \,\frac{t^{4\ell+2} F^{(2\ell+1)}(-t^2)}{(2\ell+1)!}\, \binom{2\ell+1}{\ell+1+m}\,p^{\ell+1+m}(1-p)^{\ell-m}\cdot \mathbbm{1}_{-(\ell+1)\leq m\leq \ell}, &\quad j = i+2m+1,\,m\in\mathbb{Z}.
\end{split}
\right.
\]

$\text{ where } \, \, F^{(k)} (x) \, :=\, \frac{\,{\mathrm d}^{k}  F \,}{\,{\mathrm d} x^{k}\,}(x)  \, ; \quad x > 0 \,, \, \, k \in \mathbb N \,  . $
Here the $\,k\,$-th derivative $\, F^{(k)}(x) \,$ of $\, F(\cdot) \,$ can be written as $\,F^{(k)}(x) \, =\,  \rho_{k}(x) e^{ - \sqrt{-x}}\,$, where 
\[
\begin{split}
\rho_k(x)&=\frac{1}{2^k} \sum\limits_{j=k}^{2k-1}\, \frac{(j-1)!}{(2j-2k)!!(2k-j-1)!} \, (-x)^{\,-\frac{j}{2}},\quad \textit{for} \quad k\geq 1,
\end{split}
\]
and $\rho_{0}(x) \, =\,  1$ for $x \le 0 $.  

Thus the Gaussian process $\, X_t^0 \,$,  $\, t \ge 0 \,$, corresponding to the Markov chain, is 
\begin{equation}\label{gaussian}
\begin{split}
{X}_t^0 \, : & =\, \sum\limits_{j=-\infty}^{\infty}\int^{t}_{0}   (\exp ( Q (t-s)))_{0,j} {\mathrm d} {W}_s^j \, \\
&=\, \sum\limits_{j\,\, even}\int^{t}_{0}   (\exp ( Q (t-s)))_{0,j} {\mathrm d} {W}_s^j \, +\, \sum\limits_{j\,\, odd}\int^{t}_{0}   (\exp ( Q (t-s)))_{0,j} {\mathrm d} {W}_s^j \, \\
&=\sum\limits_{m=-\infty}^\infty \int^{t}_{0}\sum\limits_{\ell=0}^\infty  \,\frac{(t-s)^{4\ell} F^{(2\ell)}(-(t-s)^2)}{(2\ell)!}\,\binom{2\ell}{\ell+m}\,p^{\ell+m}(1-p)^{\ell-m}\cdot \mathbbm{1}_{-\ell\leq m\leq \ell}\,{\mathrm d} {W}_s^{2m}\\
&+\sum\limits_{m=-\infty}^\infty \int^{t}_{0}\sum\limits_{\ell=0}^\infty  \,\frac{(t-s)^{4\ell+2} F^{(2\ell+1)}(-(t-s)^2)}{(2\ell+1)!}\, \binom{2\ell+1}{\ell+1+m}\,p^{\ell+1+m}(1-p)^{\ell-m}\cdot \mathbbm{1}_{-(\ell+1)\leq m\leq \ell}\, {\mathrm d} {W}_s^{2m+1}\\
\\
&=\sum\limits_{\ell=0}^\infty\sum\limits_{m=-\ell}^\ell \,\int^{t}_{0}\frac{(t-s)^{4\ell} F^{(2\ell)}(-(t-s)^2)}{(2\ell)!}\,\binom{2\ell}{\ell+m}\,p^{\ell+m}(1-p)^{\ell-m}{\mathrm d} {W}_s^{2m}\\
&+\sum\limits_{\ell=0}^\infty \sum\limits_{m=-(\ell+1)}^\ell\,\int^{t}_{0} \frac{(t-s)^{4\ell+2} F^{(2\ell+1)}(-(t-s)^2)}{(2\ell+1)!}\, \binom{2\ell+1}{\ell+1+m}\,p^{\ell+1+m}(1-p)^{\ell-m}  {\mathrm d} {W}_s^{2m+1}\,,
\end{split}
\end{equation}
where $\, {W}_.^k (\cdot) \,$, $\,  k \in \mathbb Z\,$ are independent standard Brownian motions. 

Thus, the variance is given by 
\begin{equation} 
\begin{split}
\text{Var} ( X_{t}^{0}) \,
&= \text{Var} \big(\sum\limits_{\ell=0}^\infty\sum\limits_{m=-\ell}^\ell \int^{t}_{0} \frac{(t-s)^{4\ell} F^{(2\ell)}(-(t-s)^2)}{(2\ell)!}\,\binom{2\ell}{\ell+m}\,p^{\ell+m}(1-p)^{\ell-m}\,{\mathrm d} {W}_s^{2m}\big)\\
&+ \text{Var}  \big(\sum\limits_{\ell=0}^\infty \sum\limits_{m=-(\ell+1)}^\ell \int^{t}_{0} \frac{(t-s)^{4\ell+2} F^{(2\ell+1)}(-(t-s)^2)}{(2\ell+1)!}\, \binom{2\ell+1}{\ell+1+m}\,p^{\ell+1+m}(1-p)^{\ell-m} \, {\mathrm d} {W}_s^{2m+1}\big).
\end{split}
\end{equation}

\bibliographystyle{acm.bst}

\end{document}